\documentclass[11pt,letterpaper]{amsart}

\numberwithin{equation}{section}
\theoremstyle{plain}
\newtheorem{theorem}{Theorem}[section]
\newtheorem{prop}[theorem]{Proposition}

\newtheorem{lemma}[theorem]{Lemma}

\newtheorem{corollary}[theorem]{Corollary}

\theoremstyle{definition}

\usepackage{hyperref}

\newtheorem{remark}[theorem]{Remark}

\newcommand{\Rmnum}[1]{\expandafter\@slowromancap\romannumeral #1@}

\newcommand{\mr}{\mathbb{R}}
\newcommand{\ud}{\mathrm{d}}
\newcommand{\ms}{\mathbb{S}}
\newcommand{\fint}{-\mkern -19mu\int}
\allowdisplaybreaks

\address{Mingxiang Li, Department  of Mathematics \& Institue of Mathematical Sciences, Chinese University of Hong Kong, Shatin, NT, Hong Kong}
\email{mingxiangli@cuhk.edu.hk}

\address{Juncheng Wei,  Department of Mathematics, Chinese University of Hong Kong, Shatin, NT, Hong Kong}
\email{wei@math.cuhk.edu.hk}

\address{X. Xu,   School  of Mathematics, Nanjing University, China }
\email{matxuxw@nju.edu.cn}
\begin{document}
	\title{On geometry of $Q^{(2k)}_g$-curvature}
	\author{Mingxiang Li, Juncheng Wei, Xingwang Xu}
	\date{}
	\maketitle
	\begin{abstract}
		The main purpose of current article is to study the geometry of $Q$-curvature.  For simplicity, we start with a simple model: a complete and conformal  metric $g=e^{2u}|dx|^2$ on $\mathbb{R}^n$.  Assuming that the metric $g$ has non-negative $nth$-order $Q$-curvature and non-negative scalar curvature,  we show that the Ricci curvature is non-negative.  If we further assume that  the isoperimetric ratio near the end is positive,  we show that the growth rate of $kth$ elementary symmetric function $\sigma_k(g)$ of Ricci curvature over geodesic ball  of  radius $r$ is at most polynomial in $r$  with order $n-2k$ for all $1 \leq k \leq \frac{n-2}{2}$.  Similarly, we are able to show that the same growth control holds for $2kth$-order $Q$-curvature. Finally, we  show  that for $k=1$ or $2$, the gap theorems for $Q^{(2k)}_g$ hold true.
	\end{abstract}
	
	\section{Introduction}
	
	The geometry of open manifolds is one of  central topics in differential geometry, presenting numerous challenges.  A very famous one is the following Yau's conjecture:  given a complete Riemannian manifold $(M^n,g)$ of dimension at least three with non-negative Ricci curvature, the scalar curvature near the end should have suitable decay, more precisely:
	\begin{equation}\label{Yau's conjecture}
		\limsup_{r\to\infty} r^{2-n} \int_{B^g_r(p_0)} R_g dv_g < +\infty
	\end{equation}
	where $R_g$ is the scalar curvature of metric $g$ and $B^g_r(p_0)$ represents the geodesic ball centered at point $p_0$  with radius  $r$.  This conjecture can be viewed as a higher-dimensional analogue of the classical Cohn-Vossen inequality \cite{CV} which says the total Gaussian curvature of an open surface is finite if the Gaussian curvature is non-negative.  Higher dimensional geometry is much more complicated. This conjecture provides one possible way to study the higher dimensional geometry of the complete open manifolds.
	
	So far, the only progress on resolving this conjecture is made through the works of \cite{Xu gy} and \cite{Zhu} in very special case: three dimensional manifolds with one pole.  See also \cite{M-W} for the sharp and precise upper bound of  \eqref{Yau's conjecture} under the assumption that $n=3$ and the scalar curvature is bounded between two positive constants. One  can also refer to \cite{Pet} for related result under the control of sectional curvatures.
	In fact, the original conjecture of Yau \cite{Yau} is more general and more ambitious:
	\begin{equation}\label{Yau's conjecture for sigma_k}
		\limsup_{r\to\infty} r^{2k-n} \int_{B^g_r(p_0)} \sigma_k dv_g < +\infty
	\end{equation}
	where $\sigma_k$ is the $kth$ elementary symmetric function of the Ricci tensor.  
	However, for $k\geq2$, Yang \cite{Yang} has shown that the conjecture is false in general by construction of some counterexamples in the case of  $U(n)$-invariant complete K\"ahler metrics on $\mathbb{C}^n$ with non-negative bisectional curvature.
	
	The conjecture has been existed more than three decades and very limited progress has been made. To rethink of this problem, we observe that Cohn-Vossen's work \cite{CV}  can also be thought of the conformal geometry since it is two dimensional case. This motivated us to consider it from conformal geometric point of view.
	
	Recall that, for manifolds with dimension greater than or equal three, T. Branson introduced the concept of $Q$-curvature which is a natural analogue of Gaussian curvature. Let us take this point of view and try to check Yau's conjecture. Thus, we are reaching at the position to consider the conformal metric on $\mathbb{R}^n$.  Let us consider the conformal metric $g=e^{2u}|dx|^2$ on $\mathbb{R}^n$ with $n\geq 4$ an even integer, for brevity, the  metrics are assumed to be smooth. Although, in general, the definition of $Q$-curvature is very complicated, the case we consider here is much simpler: the top-order $Q$-curvature  $Q^{(n)}_g$ in our case is defined by the following  conformally  invariant equation
	\begin{equation}\label{nth order Q}
		(-\Delta)^{\frac{n}{2}}u=Q^{(n)}_ge^{nu}.
	\end{equation}
	Again due to the conformal structure, lower order $Q$-curvatures also take the simple form: for each integer $1\leq k<\frac{n}{2}$, $2kth$-order $Q$-curvature $Q^{(2k)}_g$ can be computed through the equation:
	\begin{equation}\label{Q_g^2k}
		(-\Delta)^{k}e^{\frac{n-2k}{2}u}=Q^{(2k)}_ge^{\frac{n+2k}{2}u}.
	\end{equation}
	In particular, $Q^{(2)}_g$ is exactly the scalar curvature $R_g$ multiplying $\frac{n-2}{4(n-1)}$. See \cite{LX} for more discussions on  $Q^{(2k)}_g$-curvatures on $\mr^n$.
	
	There are a lot of works on this type of geometry: for the case $n=4$, Chang, Qing, and Yang \cite{CQY} first established a sharp upper bound for the $4th$-order $Q$-curvature integral under the assumptions that $Q^{(4)}_ge^{4u}\in L^1(\mathbb{R}^4)$ and the scalar curvature is non-negative outside a compact set. Subsequently, for dimensions $n\geq 4$, Fang \cite{Fa} and Nidaye-Xiao \cite{NX} obtained analogous results under similar assumptions, namely $Q^{(n)}_ge^{nu}\in L^1(\mathbb{R}^n)$ and $R_g \geq 0$ outside a compact set.
	
	Our primary objective is to demonstrate that the $Q$-curvature integral can be controlled from above under appropriate geometric conditions, analogous to the Cohn-Vossen inequality. In this context, the assumption $Q^{(n)}_ge^{nu}\in L^1(\mathbb{R}^n)$ appears somewhat too restrictive. 
	
	Let us take step back: recall, for the two-dimensional open surfaces, Huber \cite{Hu} extended the Cohn-Vossen inequality under the weaker assumption that only the negative part of the Gaussian curvature is integrable. In dimension four, Chang, Hang, and Yang \cite{CHY} made a significant improvement over the result of Chang-Qing-Yang \cite{CQY,CQY Inv} by requiring the weaker condition: the negative part of $Q^{(4) }_g$ is integrable over the manifold, in other words, $(Q^{(4)}_g)^-e^{4u}\in L^1(\mathbb{R}^4)$, where we define $\varphi^-:=\max\{-\varphi, 0\}$, under  extra assumptions 
	\begin{equation}\label{additional assmptions}
		|Ric_g|+|\nabla R_g|\leq C, \quad R_g\geq C>0
	\end{equation}
(see Thorem 1.5 in \cite{CHY}).  This condition allows for the application of the Cheng-Yau estimate. For general conformally  compact manifolds, very recently, Ma and Qing (See \cite{MQ adv}, \cite{MQ APDE}) have taken some important steps for understanding $4th$-order $Q$-curvature by employing potential theory. In particular, for $n=4$, they relax the additional assumptions \eqref{additional assmptions} to $R_g\geq 0$ outside a compact set  (see Theorem 1.5 in \cite{MQ APDE}).
	
As a first step toward understanding our problem, we aim to extend the $Q$-curvature integral control established in dimension four by \cite{CHY, MQ APDE} to arbitrary even dimensions $n$. However, the methods presented in these references do not readily extend to higher order $Q$-curvature, even in the conformally flat case. This necessitates a new technique that remains effective irrespective of the dimension. Our first principal result is the following.
	
	\begin{theorem}\label{thm: int-Q^n-between 0,1}
		Let $g=e^{2u}|dx|^2$ be a complete and conformal metric on $\mr^n$  with $n\geq 4$ an even integer.  Suppose that its scalar curvature $R_g$ is non-negative outside a compact set and the negative part of $nth$-order $Q^{(n)}_g$-curvature is integrable over $(\mr^n,g)$. Then, the total $Q^{(n)}_g$-curvature is controlled precisely by 
		\begin{equation}\label{CVineqaulity for Q}
			0\leq \int_{\mr^n}Q^{(n)}_g\ud \mu_g\leq \frac{(n-1)!|\mathbb{S}^n|}{2}
		\end{equation}
		where $|\mathbb{S}^n|$ denotes the volume of standard n-sphere and $\ud\mu_g=e^{nu}\ud x$.
	\end{theorem}
	
	In their work \cite{Guan-V-W}, Guan, Viaclovsky, and Wang have observed the non-negativity of Ricci curvature if the $k$-th elementary symmetric function of the Schouten tensor is non-negative for all $1 \leq k \leq n/2$, leveraging the algebraic structure of the elementary symmetric functions. This foundational result naturally motivates the investigation of Ricci curvature control through the lens of $Q$-curvature positivity, as explored in \cite{LX}. In the present paper, we demonstrate that the non-negativity of both $nth$-order $Q$-curvature and scalar curvature suffices to guarantee the non-negativity of Ricci curvature.
	\begin{theorem}\label{thm: positive Ricci}
		Let $g=e^{2u}|dx|^2$ be a complete and conformal metric on $\mr^n$  where $n\geq 4$ is an even integer.  Suppose  the  $nth$-order $Q^{(n)}_g$-curvature  and scalar curvature are non-negative. Then, the Ricci curvature $Ric_g$ is non-negative.
	\end{theorem}

	Building on Theorem \ref{thm: positive Ricci}, we observe that the positivity condition on $Q$-curvatures imposes stronger constraints compared with Ricci curvature. This motivates our reconsideration of Yau’s conjecture \eqref{Yau's conjecture for sigma_k}. However, due to technical limitation, we must additionally assume that the isoperimetric ratio $I_g$ is positive. 
	Here, the isoperimetric ratio $I_g$ is defined as
	\begin{equation}\label{def:isoperimetric_ratio}
		I_g := \inf_{\Omega \subset \mathbb{R}^n} \frac{\operatorname{Vol}_g(\partial \Omega)^{\frac{n}{n-1}}}{\operatorname{Vol}_g(\Omega)},
	\end{equation}
	where $\Omega$ ranges over all smooth bounded domains in $\mathbb{R}^n$.
	\begin{theorem}\label{thm:Yau's conjecture}
		Under the same assumptions as in Theorem \ref{thm: positive Ricci}, if  we further assume that  the  isoperimetric ratio $I_g$ is positive, 
		then, for any integer $1\leq k\leq \frac{n-2}{2}$, Yau's conjecture \eqref{Yau's conjecture for sigma_k} holds.
	\end{theorem}

	Interestingly, Yau's conjecture  \eqref{Yau's conjecture for sigma_k} leads us to consider the growth behavior of lower-order $Q$-curvature integrals.
	\begin{theorem}\label{thm:Q^2k_g growth_on_geodesic_ball}
		Under the same assumptions as in Theorem \ref{thm: int-Q^n-between 0,1}, if we also assume that  the  isoperimetric ratio $I_g$ is positive, then, for any fixed point $p_0$ and any integer $1\leq k\leq \frac{n-2}{2}$,  there holds
		\begin{equation}\label{Q2k integeral growth}
			\limsup_{r\to\infty}r^{2k-n}\int_{B_r^g(p_0)}|Q_g^{(2k)}|\ud\mu_g<+\infty
		\end{equation}
		where $B^g_r(p_0)$ represents the geodesic ball centered at point $p_0$  with radius  $r$.
	\end{theorem}
	\begin{remark}
		We conjecture that the condition $I_g>0$ is unnecessary.
	\end{remark}
	
It is natural to ask whether the integral growth rate in \eqref{Q2k integeral growth} is optimal.
Inspired by various gap theorems—such as those in \cite{Chen-Zhu, GPZ, Greene-Wu, MSY, Ni}, among others—we prove the following gap results for the $Q^{(2k)}_g$-curvature.
Due to technical limitations, our results are currently restricted to the cases $k = 1, 2$. We conjecture that such gap theorems hold for all integers $k$ satisfying $1 \leq k \leq \frac{n-2}{2}$.

	\begin{theorem}\label{thm:Q^2k_g gap theorem}
		Let $g=e^{2u}|dx|^2$  be a complete and conformal metric on $\mr^n$ where $n\geq 4$ is an even integer.  Suppose both the scalar curvature $R_g$ and $nth$-order $Q$-curvature  are non-negative as well as the isoperimetric ratio $I_g$ is positive. If, for some given point $p_0$, an integer $1\leq k\leq\min\{ \frac{n-2}{2},2\}$ and some small constant $\epsilon>0$ such that 
		$$\limsup_{r\to\infty}r^{2k-n+\epsilon}\int_{B_r^g(p_0)}|Q_g^{(2k)}|\ud\mu_g<+\infty,$$
		then $u$ must be a constant.
	\end{theorem}
	
	If we give a more restrictive growth assumption, we are able to including the case $I_g=0.$
	
	\begin{theorem}\label{thm:Q^2k_g gap theoremfor logr}
		Let $g=e^{2u}|dx|^2$  be a complete and conformal metric on $\mr^n$ where $n\geq 4$ is an even integer.  Suppose both the scalar curvature $R_g$ and $nth$-order $Q$-curvature  are non-negative. If, for some given point $p_0$, an integer $1\leq k\leq\min\{ \frac{n-2}{2},2\}$ such that 
		$$\limsup_{r\to\infty}(\log r)^{-1}\int_{B_r^g(p_0)}|Q_g^{(2k)}|\ud\mu_g<+\infty,$$
		then $u$ must be a constant.
	\end{theorem}

	This paper is organized as follows. In Section \ref{sec:log potential}, we establish key estimates for logarithmic potentials, which serve as foundational tools for our subsequent analysis. Section \ref{sec:non-negative R_g} is devoted to investigation of the growth properties of the function $u$  under appropriate scalar curvature assumptions, culminating in the proof of Theorem \ref{thm: int-Q^n-between 0,1}.  In Section \ref{sect:positive-Q_4}, we discuss the positivity of lower order $Q$-curvatures and Ricci curvature.  As  an application, we finish the proof of Theorem \ref{thm: positive Ricci}. Section  \ref{sec:A_inftyweight} provides essential background on strong $A_\infty$ weights, preparing for our final results. In Section \ref{sec: proof of integral growth}, we present the proofs of Theorems \ref{thm:Yau's conjecture} and  \ref{thm:Q^2k_g growth_on_geodesic_ball}.
	In Section \ref{sec:gap theorem}, we discuss the gap theorems for $Q^{(2k)}_g$ and prove Theorem \ref{thm:Q^2k_g gap theorem} and Theorem \ref{thm:Q^2k_g gap theoremfor logr}. Finally, we give some remarks on the vanishing isoperimetric ratio case in Section \ref{sec:vanishing I_g}.

	\hspace{3em}
	
	{\bf Acknowledgment.} The first author would like to thank Fang Wang for helpful discussions and for  bringing  the work \cite{Wang Fang} to his attention. The second author  is partially supported  by   Hong Kong General Research Fund “New frontiers in singular limits of nonlinear partial differential equations".   The third author is supported by NSFC (No.12171231).

	\section{Logarithmic potential estimates and normal metrics}\label{sec:log potential}

	In this section, we are going to establish some crucial estimates for logarithmic potential  to be used later. Given a function $f(x)\in L^\infty_{loc}(\mr^n)\cap L^1(\mr^n)$ with even integer $n\geq 2$. It is well known that the logarithmic potential 
	\begin{equation}\label{log potential}
		\mathcal{L}(f)(x):=\frac{2}{(n-1)!|\mathbb{S}^n|}\int_{\mr^n}\log\frac{|y|}{|x-y|}f(y)\ud y,
	\end{equation}
	satisfies the differential equation
	\begin{equation}\label{(-delta)^n/2 L(f)}
		(-\Delta)^{\frac{n}{2}}\mathcal{L}(f)(x)=f(x).
	\end{equation}
	More information on the logarithmic potential can be founded in Section 2 of  \cite{Li 23 Q-curvature}. For easier presentation, we list several notations we are going to use:
	first, we use $\alpha$ to indicate the some kind normalization of the total curvature, namely,
	$$\alpha:=\frac{2}{(n-1)!|\mathbb{S}^n|}\int_{\mr^n}f(x)\ud x.$$ Secondly, we should also use the average integration as usual: for given  a measurable set $E$, the average integral for function $\varphi$ over $E$  will be written as
	$$\fint_E\varphi\ud \mu:=\frac{1}{|E|}\int_E\varphi\ud \mu.$$
	The third notation is that $\ud\sigma(x)$ stands for the standard measure on sphere induced from Euclidean metric. 
	In particular, set the notation  $$\bar \varphi(r):=\fint_{\partial B_r(0)}\varphi\ud\sigma.$$ We use $C$ to denote a positive constant which may be different from line to line throughout this paper.
	
	Now,  we start up our potential estimates for a given function $f(x)\in L^\infty_{loc}(\mr^n)\cap L^1(\mr^n)$. As a first step, we try to set the growth rate of the potential.
	\begin{lemma}\label{lem:barLf}(See Lemma 2.5 in \cite{Li 24 total Q-curvature})
		There holds
		$$\lim_{r\to\infty}\frac{\overline{\mathcal{L}(f)}(r)}{\log r}=-\alpha.$$
	\end{lemma}

	The  following two lemmas are elementary but very useful.
	\begin{lemma}\label{lem: mean value}(See Lemma 2.1 in \cite{Li 24 total Q-curvature})
		If $0< p\leq n-2$ and $r>0$, then for all $y \in \mr^n$, there holds
		$$\fint_{\partial B_r(0)}\left(\frac{|y|}{|x-y|}\right)^p\ud\sigma(x)\leq 1.$$
	\end{lemma}
	
	\begin{lemma}\label{lem:int 1/|x-y|^p leq r^-p}
		For  $0< p< n-1$ and $r > 0$, there holds
		$$\fint_{\partial B_r(0)}\frac{1}{|x-y|^p}\ud\sigma(x)\leq C(n,p)r^{-p}$$
		for all $y \in \mr^n$ with $n \geq 3$ where $C(n,p)$ is a constant depending on $n$ and $p$.
	\end{lemma}
	\begin{proof}
		If $ y = 0$, the inequality trivially holds true. For $|y|>0$, we define a function 
		$$ h(y, r, p) = \fint_{\partial B_r(0)}\frac{1}{|x-y|^p}\ud\sigma(x).$$  Notice that $h(0,r,p) = r^{-p}$ and $\lim_{|y| \to \infty} h(y,r,p) = 0$. It is easy to see that, if $n-2\leq p < n-1$, $h$ is sub-harmonic function for $y$ and for each fixed $r$. By the application of the maximum principle for subharmonic function on the domain $|y| < r$ as well as $|y| > r$, we see that the maximum value of $h$ can only be achieved on the sphere ${\partial B_r(0)}$. However, if $|y| = r$, by scaling, we have $h(y, p, r) \leq h(y, 1, p) r^{-p}$. Notice that $h(y, 1, p)$ is a constant which will be denoted by $C(n, p) = h(y, 1, p)$, our result follows in this case. Now if $0 < p < n-2$, then $q = \frac{n-2}{p} > 1$, by H\"older's inequality, we have $ h(y, r, p)^q \leq h(y, r, n-2) \leq h(y, 1, n-2) r^{2-n}$ by application of what we just proved. Therefore, we can easily obtain that $h(y, r, p) \leq (h(y, 1, n-2))^{\frac{p}{n-2}} r^{-p}$ which is what we are going to show with $C(n, p) = h(y, 1, n-2)^{\frac{p}{n-2}}$. The proof of lemma is complete.
		
	\end{proof}

For the case $n=4$, the following lemma is essentially identical to Lemma 3.2 in \cite{CQY}; a slightly different proof can be found in Lemma 2.7 of \cite{Li 24 total Q-curvature} for general $n$.
	\begin{lemma}\label{lem:e^ku=ek bar u}(See Lemma 2.7 in \cite{Li 24 total Q-curvature})
		For any $k>0$,   there holds
		$$\fint_{\partial B_r(0)}e^{k\mathcal{L}(f)}\ud \sigma=e^{k\overline{\mathcal{L}(f)}(r)+o(1)}$$
		where $o(1)\to 0$ as $r\to\infty.$
	\end{lemma}

	Building upon the above lemma, we now are in the position to establish a crucial estimate for the logarithmic potential involving higher-order derivatives.
	\begin{lemma}\label{lem:equ nabla^k Lf}
		Given real numbers $0<p<n-1$, $q>0$ and an integer $1\leq k\leq n-2$.   The following estimate holds true for all $r$ sufficiently large
		$$\fint_{\partial B_r(0)}e^{q\mathcal{L}(f)}|\nabla^k \mathcal{L}(f)|^{\frac{p}{k}}\ud \sigma\leq C r^{-p}e^{q\overline{\mathcal{L}(f)}(r)}.$$
	\end{lemma}
	
	\begin{proof}
		
		For any $p_1$ satisfying $k \leq p_1<n-1$, H\"older's inequality implies that 
		\begin{align*}
			|\nabla^k \mathcal{L}(f)(x)|^{\frac{p_1}{k}}\leq &C\left(\int_{\mr^n}\frac{1}{|x-y|^k}|f(y)|\ud y\right)^{\frac{p_1}{k}}\\
			\leq &C\left(\int_{\mr^n}\frac{1}{|x-y|^{p_1}}|f(y)|\ud y\right)\left(\int_{\mr^n}|f(y)|\ud y\right)^{\frac{p_1-k}{k}}\\
			\leq &C\int_{\mr^n}\frac{1}{|x-y|^{p_1}}|f(y)|\ud y.
		\end{align*}
		Applying  Lemma \ref{lem:int 1/|x-y|^p leq r^-p} and Fubini's theorem, one has
		\begin{equation}\label{fintn nabla u^p}
			\fint_{\partial B_r(0)}|\nabla^k \mathcal{L}(f)|^{\frac{p_1}{k}}\ud\sigma\leq Cr^{-p_1}.
		\end{equation}
		The estimate  \eqref{fintn nabla u^p}, together with Lemma \ref{lem:e^ku=ek bar u}, implies
		\begin{align*}
			&\fint_{\partial B_r(0)}e^{q\mathcal{L}(f)}|\nabla^k \mathcal{L}(f)|^{\frac{p}{k}}\ud\sigma\\
			\leq &\left(\fint_{\partial B_r(0)}e^{\frac{p_2q}{p_2-p}\mathcal{L}(f)}\ud\sigma\right)^{1-\frac{p}{p_2}}\left(\fint_{\partial B_r(0)}|\nabla^k \mathcal{L}(f)|^{\frac{p_2}{k}}\ud\sigma\right)^{\frac{p}{p_2}}\\
			\leq &Ce^{q\overline{\mathcal{L}(f)}(r)}r^{-p}
		\end{align*}	
		where the constant $p_2$ is chosen to satisfy $\max\{k,p\}< p_2<n-1$.
		
		Thus, the  proof is complete.
	\end{proof}
	
	Now, let us recall some facts. Usually, a conformal metric $g = e^{2u}|dx|^2$ on $\mathbb{R}^n$ has \emph{finite total $Q^{(n)}_g$-curvature} if 
	$$
	\int_{\mathbb{R}^n} |Q^{(n)}_g| e^{nu} \, dx < +\infty.
	$$
	Such a metric is called \emph{normal metric} if it satisfies
	\begin{equation}\label{normal-metric-def}
		u(x) = \mathcal{L}(Q^{(n)}_g e^{nu})(x) + C.
	\end{equation}

	The concept of normal metrics was firstly introduced by Finn \cite{Finn} for complete open surfaces and later extended to higher dimensions with higher-order curvature by Chang, Qing, and Yang \cite{CQY}. An important geometric criterion for normal metrics was established in \cite{CQY}, where the authors prove that for a conformal metric with finite total $Q^{(n)}_g$-curvature, if the scalar curvature $R_g \geq 0$ outside some compact set, then the metric must be normal. 
	
	Let us further recall that the total $Q^{(n)}_g$-curvature integral is just the Chern-Gauss-Bonnet integral, it should measure the topology of $\mr^n$. The authors of~\cite{CQY} also show that the deficit between  $Q^{(n)}_g$-curvature integral  and the topology of $\mr^n$ is measured by an isoperimetric ratio at infinity. Notice that the Euler number of $\mr^n$ is equal to $1$. To be more precise, we repeat it here for reader's convenience.
	\begin{theorem}\label{thm:CQY-thm}(See \cite{CQY}, \cite{Fa}, \cite{NX})
		Suppose a complete and conformal metric $g=e^{2u}|dx|^2$ on $\mr^n$ with finite total $Q^{(n)}_g$-curvature with $n\geq 4$  an even integer is given. If the scalar curvature $R_g$ is non-negative outside  a compact set, then there holds
		$$1-\frac{2}{(n-1)!|\mathbb{S}^n|}\int_{\mr^n}Q^{(n)}_ge^{nu}\ud x=\lim_{r\to\infty}\frac{V_g(\partial B_r(0))^{\frac{n}{n-1}}}{c_nV_g(B_r(0))}$$
		where $c_n$ is a positive constant depending on $n$, to be precise, $$c_n = |\mathbb{S}^{n-1}|^{\frac{n}{n-1}} |\mathbb{S}^n|^{-1}.$$
	\end{theorem}

	An immediate corollary of the above theorem yields the following result concerning the isoperimetric ratio. For more substantial developments regarding $Q$-curvature and isoperimetric inequalities, we refer to \cite{Wang IMRN}, \cite{Wang Yi}.
	\begin{corollary}\label{cor:I_g-positive}
		Given a complete and conformal metric $g=e^{2u}|dx|^2$ on $\mr^n$ with finite total $Q^{(n)}_g$-curvature where $n\geq 4$ is an even integer. If the scalar curvature $R_g\geq 0$ outside  a compact set and the isoperimetric ratio $I_g$ is positive, then there holds
		\begin{equation}\label{alpha_0<1}
			\int_{\mr^n}Q^{(n)}_ge^{nu}\ud x<\frac{(n-1)!|\mathbb{S}^n|}{2}.
		\end{equation}
	\end{corollary}
	\begin{proof}
		Based on the definition of $I_g$ \eqref{def:isoperimetric_ratio} and $I_g>0$, one has
		$$\lim_{r\to\infty}\inf \frac{V_g(\partial B_r(0))^{\frac{n}{n-1}}}{V_g(B_r(0))}\geq I_g>0.$$
		Using Theorem \ref{thm:CQY-thm}, we finish the proof.
	\end{proof}
	\begin{remark}
		It follows from Theorem 1.1 in \cite{Wang IMRN} that the isoperimetric ratio $I_g$ is positive provided that \eqref{alpha_0<1} holds.
	\end{remark}

	\section{Conformal metrics with non-negative scalar curvature}\label{sec:non-negative R_g}

	In this section, we analyze the asymptotic behavior near the ends for conformal metrics with non-negative scalar curvature. For related results, we refer the interested reader to the references \cite{CHY}, \cite{SY Inv}, and \cite{MQ APDE}.

	\begin{prop}\label{prop:baru-upper-bound}
		Consider a conformal metric $g=e^{2u}|\ud x|^2$ on $\mr^n$ with  $n\geq 3$.	If the negative part of scalar curvature $R_g$ of the metric $g$ is in $L^{\frac{2n}{n+2}}(\mr^n,g)$, then, the spherical average of $u^+$ where $u^+:=\max\{u,0\}$ is bounded above, i.e., there exists a constant $C > 0$ such that
		$$\overline{u^+}(r)\leq C.$$
	\end{prop}
	\begin{proof}
		Firstly, a direct computation shows that the scalar curvature $R_g$ satisfies the following equation
		\begin{equation}\label{scalar curvature}
			(-\Delta)e^{\frac{n-2}{2}u}=\frac{n-2}{4(n-1)}R_ge^{\frac{n+2}{2}u}.
		\end{equation}
		Applying  the divergence theorem and H\"older's inequality, we obtain, for any $r \geq 1$, that 
		\begin{align*}
			&\fint_{\partial B_{r}(0)}e^{\frac{n-2}{2}u}\ud \sigma-\fint_{\partial B_{1}(0)}e^{\frac{n-2}{2}u}\ud \sigma\\
			=&\int^r_{1}\frac{\ud }{\ud s}\left(\fint_{\partial B_s(0)} e^{\frac{n-2}{2}u}\ud\sigma\right)\ud s\\
			=&\int^r_{1}\fint_{\partial B_s(0)}\frac{\partial e^{\frac{n-2}{2}u}}{\partial \nu}\ud\sigma\ud s\\
			=&\int^r_{1}|\partial B_s(0)|^{-1}\int_{B_s(0)}\Delta e^{\frac{n-2}{2}u}\ud x\ud s\\
			\leq &\frac{n-2}{4(n-1)}\int^r_{1}|\partial B_s(0)|^{-1}\int_{B_{s}(0)}R_g^-e^{\frac{n+2}{2}u}\ud x\ud s\\
			\leq &\frac{n-2}{4(n-1)}\int^r_{1}|\partial B_s(0)|^{-1}\left(\int_{B_{s}(0)}(R_g^-)^{\frac{2n}{n+2}}e^{nu}\ud x\right)^{\frac{n+2}{2n}}|B_s(0)|^{\frac{n-2}{2n}}\ud s\\
			\leq &C(n)\|R_g^-\|_{L^{\frac{2n}{n+2}}(\mr^n,g)}\int^r_{1}s^{-\frac{n}{2}}\ud s\\
			\leq &C(n)\|R_g^-\|_{L^{\frac{2n}{n+2}}(\mr^n,g)}
		\end{align*}
		where $\nu$ is the out normal  derivative and   the assumption $n\geq 3$ is used in the last estimate.	Then, using the fact $e^t\geq \max\{t,0\}$, there holds
		$$	C\geq \fint_{\partial B_{r}(0)}e^{\frac{n-2}{2}u}\ud \sigma\geq \frac{n-2}{2}\fint_{\partial B_{r}(0)}u^+\ud\sigma$$
		which yields that,  for $r\geq 1$,
		\begin{equation}\label{bar u upper bounded by C}
			\overline{u^+}(r)\leq C.
		\end{equation}
		Obviously,  for $0\leq r<1$, there holds
		$$\overline{u^+}(r)\leq \|u\|_{L^\infty(B_1(0))}.$$
		Thus, we finish the proof.
	\end{proof}
	
	We now turn to the lower bound estimate of $u$ with the help of maximum principle.
	\begin{prop}\label{prop:bar u lower bound}
		Consider a conformal metric $g=e^{2u}|dx|^2$ on $\mr^n$ with $n\geq 3$. Supposing that the scalar curvature $R_g\geq 0$ outside a compact set, then there holds
		$$u(x)\geq -2\log(|x|+1)-C.$$
	\end{prop}
	\begin{proof}
	Using  \eqref{scalar curvature} and our assumptions, there exists  $R_0>0$ such that for any $|x|\geq R_0$, one has
		$$\Delta e^{\frac{n-2}{2}u}\leq 0.$$
	For $|x|\geq R_0$, consider the function $$w_0(x):=e^{\frac{n-2}{2}u(x)}-c_1|x|^{2-n}$$
	where $c_1=R_0^{n-2}\min_{|x|=R_0}e^{\frac{n-2}{2}u(x)}$. Thus, one has $w_0(x)\geq 0$ on $\partial B_{R_0}(0)$. It is not hard to see that $$\lim_{|x|\to\infty}\inf w_0(x)\geq 0.$$ Thus, with the help of maximum principle, one must have 
	$$w_0(x)\geq 0, \quad |x|\geq R_0$$
	which yields that 
	$$u(x)\geq -2\log(|x|+1)-C.$$
	\end{proof}

In fact, for a complete metric, a lower bound for $u$ can be established under weaker assumptions, such as the scalar curvature being bounded below or the negative part of the scalar curvature having finite $L^p$ norm. For details, see Proposition 8.1 in \cite{CHY}, Lemma 3.3 in Chapter 6 of \cite{Schoen-Yau}, and Lemma 3.1 in \cite{MQ adv}.

	Building upon Propositions \ref{prop:baru-upper-bound} and \ref{prop:bar u lower bound}, we establish a growth estimate for $|u|$, which will serve as the foundation for applying the super polyharmonic property developed in \cite{WX}.
	\begin{lemma}\label{lem:intB_r|u|}
		Given a  conformal metric $g=e^{2u}|\ud x|^2$ on $\mr^n$ where  $n\geq 3$. If the scalar curvature $R_g\geq 0$ outside  a compact set,  for $r\gg1$, there holds
		$$\fint_{B_r(0)}|u|\ud x\leq C\log r.$$
	\end{lemma}
	\begin{proof}
		Proposition \ref{prop:bar u lower bound} provides the lower bound control
		\begin{equation}\label{u^- leq Clog r}
			u^-(x)\leq 2\log(|x|+1)+C.
		\end{equation}
		Combining  the above estimate  \eqref{u^- leq Clog r} with Proposition \ref{prop:baru-upper-bound},  we derive that 
		$$\int_{B_r(0)}|u|\ud x= \int_{B_r(0)}u^+\ud x+2\int_{B_r(0)}u^-\ud x\leq Cr^n+Cr^n\log (r+1)$$
		which consequently yields that 
		$$\fint_{B_r(0)}|u|\ud x\leq C\log r, \quad r\gg 1.$$
	\end{proof}

	The following lemma is a crucial trick developed in \cite{WX}. For the reader's convenience, we sketch the proof by exactly following the argument  of Lemma 2.1 in \cite{WX}.
	
	\begin{lemma}\label{lem:WX trick}
		For any integer $p\geq 2$, consider  $v\in C^{2p}(\mr^n)$ satisfying 
		$(-\Delta)^pv\geq 0.$  Suppose that 
		\begin{equation}\label{bar B_r v=or^2}
			\fint_{B_r(0)}|v|\ud x=o(r^{2}).
		\end{equation}
		Then,  for any integer  $1\leq i\leq p-1$, we have 
		$$(-\Delta)^iv\geq 0.$$ 
	\end{lemma}
	\begin{proof}
		Set $v_i:=(-\Delta)^iv$ for $1\leq i\leq p-1$ for brevity.
		We argue by contradiction. Suppose that there exists $x_0\in \mr^n$ such that
		$$v_{p-1}(x_0)<0.$$
		Without loss of generality, one may assume that $x_0=0.$ 
		Notice that 
		$$(-\Delta)\bar v=\bar v_1, \quad (-\Delta)\bar v_{i}=\bar v_{i+1}. $$
		and
		$$(-\Delta)\bar v_{p-1}(r)=\overline{(-\Delta)^p v}(r)\geq 0$$
		which is equivalent to
		$$(r^{n-1}\bar v_{p-1}(r)')'\leq 0.$$
		Integrating it over the integer $[0,r]$, one has
		\begin{equation}\label{bar v_p-1 leq 0}
			r^{n-1}\bar v_{p-1}(r)'\leq 0.
		\end{equation}
		Using $\bar v_{p-1}(0)<0$ and \eqref{bar v_p-1 leq 0}, there holds
		$$\bar v_{p-1}(r)\leq \bar v_{p-1}(0)<0.$$
		By inductive integration, one has 
		$$(-1)^{p-1}\bar v(r)\geq c_0r^{2(p-1)}, \quad r\geq r_{p-1}$$
		for some positive constant $r_{p-1}$ and $c_0>0.$
		Then, for $r\gg1$, there holds
		$$(-1)^{p-1}\int_{B_r(0)}v\ud x=\int^r_0|\partial B_r(0)|(-1)^{p-1}\bar v(r)\ud r\geq C+Cr^{n+2(p-1)}.$$
		Using the assumptions $p\geq 2$ and \eqref{bar B_r v=or^2}, we obtain the contradiction.
		By the standard mathematical induction, we finish our proof.
	\end{proof}
	
	A Liouville type theorem for polyharmonic functions is a direct corollary of such lemma.
	\begin{corollary} \label{cor:Liouville}
		For any integer $p\geq 1$, consider $v\in C^{2p}(\mr^n)$ satisfying $\Delta^pv=0$. Suppose that 
		$$\fint_{B_r(0)}|v|\ud x=o(r).$$
		Then $v$ must be a constant.
	\end{corollary}
	\begin{proof}
		Using Lemma \ref{lem:WX trick}, one must have
		$\Delta v=0.$ Then, the classical Liouville theorem for harmonic functions deduce that $v$ must be a constant.
	\end{proof}
	\begin{remark}
		In fact, this is a very classical result and is well-known among experts. We are grateful to Q. Ngô for bringing Lemma 3.4 in their work \cite{NNPY} to our attention and for pointing out that, even when the function is  sup-polyharmonic, a Liouville-type theorem still holds under additional assumptions on $v$.
	\end{remark}
	
	As mentioned earlier, a criterion for normal metrics was established in \cite{CQY}. In what follows, we present an alternative proof of this result, based on the growth estimates for $|u|$ derived above.
	\begin{lemma}\label{lem: normal metric}
		Given a  conformal metric $g=e^{2u}|dx|^2$ on $\mr^n$ where $n\geq 4$ is an even integer with finite total $Q^{(n)}_g$-curvature. If the scalar curvature $R_g$  is non-negative outside a compact set, then $g$ is a normal metric.
	\end{lemma}
	\begin{proof}
		For brevity, 	we introduce  the notation
		$$a(x):=u(x)-\mathcal{L}(Q^{(n)}_ge^{nu})(x).$$
		By construction,  $a(x)$ is a polyharmonic function satisfying
		\begin{equation}\label{polyharmonic a}
			(-\Delta)^{\frac{n}{2}}a(x)=0.
		\end{equation}
		Applying Lemma \ref{lem:barLf},  we obtain the growth estimate
		$$	\fint_{B_r(0)}|\mathcal{L}(Q^{(n)}_ge^{nu})|\ud x\leq C\log r, \quad \text{for } r \gg 1.$$
		Combining this with Lemma \ref{lem:intB_r|u|}, one has
		\begin{equation}\label{int B_r|a|}
			\fint_{B_r(0)}|a(x)|\ud x\leq C\log r, \quad \text{for } r \gg 1.
		\end{equation}
		The growth estimate \eqref{int B_r|a|} together with the Liouville theorem for polyharmonic functions in  Corollary \ref{cor:Liouville} applied to \eqref{polyharmonic a} implies that $a(x)$ must be constant. We therefore conclude that the metric $g$ is normal.
	\end{proof}

Before completing the proof of Theorem \ref{thm: int-Q^n-between 0,1}, we first establish the following result, which does not require the metric to be complete.
		\begin{theorem}\label{thm: int-Q^n-between 0,2}
		Let $g=e^{2u}|dx|^2$ be a  conformal metric on $\mr^n$  with $n\geq 4$ an even integer.  Suppose that its scalar curvature $R_g$ is non-negative outside a compact set and the negative part of $nth$-order $Q^{(n)}_g$-curvature is integrable over $(\mr^n,g)$. Then, the total $Q^{(n)}_g$-curvature is controlled precisely by 
		\begin{equation*}
			0\leq \int_{\mr^n}Q^{(n)}_g\ud \mu_g\leq (n-1)!|\mathbb{S}^n|
		\end{equation*}
		where $|\mathbb{S}^n|$ denotes the volume of standard n-sphere and $\ud\mu_g=e^{nu}\ud x$.
	\end{theorem}
	\begin{proof}
	For any fixed radius $r_1 > 0$, we construct a smooth cut-off function $\eta_{r_1} : \mathbb{R}^n \to [0,1]$ satisfying:
	\begin{itemize}
		\item $\eta_{r_1} \equiv 1$ on the ball $B_{r_1}(0),$
		\item $\mathrm{supp} (\eta_{r_1}) \subset B_{r_1+1}(0)$.
	\end{itemize}
	Based on our assumptions, we now  introduce three key functions:
	\begin{align*}
		u_1(x) &:= \mathcal{L}\left((Q_g^{(n)})^- e^{nu}\right)(x), \\
		u_2(x) &:= \mathcal{L}\left(\eta_{r_1} (Q_g^{(n)})^+ e^{nu}\right)(x), \\
		u_3(x) &:= u(x) + u_1(x) - u_2(x).
	\end{align*}
	These definitions immediately yield the polyharmonic inequality:
	\begin{equation}\label{Delta n/2u_3 geq 0}
		(-\Delta)^{\frac{n}{2}} u_3 = (1-\eta_{r_1})(Q_g^{(n)})^+ e^{nu} \geq 0.
	\end{equation}
	Using  Propositions \ref{prop:baru-upper-bound},  \ref{prop:bar u lower bound} and Lemma \ref{lem:barLf},  one has 
	\begin{equation*}
		|\bar u_3(r)|\leq |\bar u(r)|+|\bar u_1(r)|+|\bar u_2(r)| \leq C\log (r+2).
	\end{equation*}
	As a direction computation, for $r\gg1$, there holds
	\begin{equation}\label{int B_r u_3growth}
		\int_{B_r(0)}|u_3(x)|\ud x=\int^r_0|\partial B_s(0)|\cdot|\bar u_3(s)|\ud s\leq Cr^n\log r .
	\end{equation}
	With the help of \eqref{Delta n/2u_3 geq 0} and \eqref{int B_r u_3growth}, Lemma \ref{lem:WX trick} shows that 	
	\begin{equation*}
		(-\Delta)^{i} u_3\geq 0, \quad 1\leq i\leq \frac{n}{2}-1.
	\end{equation*}
	In particular, $u_3$ is a superharmonic 	function satisfying
	$$-\Delta u_3\geq 0.$$
	Using the 	mean value property for superharmonic functions,  for all $r>0$, there holds
	\begin{equation}\label{bar u_3 upper bound}
		u_3(0)\geq\fint_{\partial B_{r}(0)}u_3\ud\sigma= \bar u_3(r).
	\end{equation}
	Notice that $u_3(0)=u(0).$	Now, using the estimate  \eqref{bar u_3 upper bound} and Lemma \ref{lem:barLf}, for $r\gg1$,  one has
	\begin{align*}
		&\bar u(r)\\
		=&\bar u_2(r)-\bar u_1(r)+\bar u_3(r)\\
		\leq &\bar u_2(r)-\bar u_1(r)+u(0).
	\end{align*}
	For $r\gg1$, dividing $\log r$ on both sides, we obtain that 
	\begin{equation}\label{bar u leq bar u_2-baru_1}
		\frac{\bar u(r)}{\log r}\leq \frac{\bar u_2(r)}{\log r}-\frac{\bar u_1(r)}{\log r}+\frac{u(0)}{\log r}.
	\end{equation}
	With the help of Proposition \ref{prop:bar u lower bound},  one has
	\begin{equation}\label{bar u inf}
		\lim_{r\to\infty}\inf\frac{\bar u(r)}{\log r}\geq -2.
	\end{equation}
	Combining these two estimates \eqref{bar u leq bar u_2-baru_1}, \eqref{bar u inf} and letting $r\to\infty$, Lemma \ref{lem:barLf} yields  that
	$$-2\leq -\frac{2}{(n-1)!|\mathbb{S}^n|}\int_{\mr^n}\eta_{r_1}(Q^{(n)}_g)^+e^{nu}\ud x+\frac{2}{(n-1)!|\mathbb{S}^n|}\int_{\mr^n}(Q^{(n)}_g)^-e^{nu}\ud x$$
	which is equivalent to
	$$\frac{2}{(n-1)!|\mathbb{S}^n|}\int_{\mr^n}\eta_{r_1}(Q^{(n)}_g)^+e^{nu}\ud x\leq \frac{2}{(n-1)!|\mathbb{S}^n|}\int_{\mr^n}(Q^{(n)}_g)^-e^{nu}\ud x+2.$$	
	By letting $r_1\to\infty$, we finally obtain that  
	$$\frac{2}{(n-1)!|\mathbb{S}^n|}\int_{\mr^n}(Q^{(n)}_g)^+e^{nu}\ud x \leq  \frac{2}{(n-1)!|\mathbb{S}^n|}\int_{\mr^n}(Q^{(n)}_g)^-e^{nu}\ud x+2$$ 
	which deduces that 
	\begin{equation*}
		\int_{\mr^n}Q_g^{(n)}e^{nu}\ud x\leq (n-1)!|\ms^n|.
	\end{equation*}
	Then,  Lemma \ref{lem: normal metric} yields that $g$ is a normal metric which means that 
	$$u(x)=\mathcal{L}(Q^{(n)}_ge^{nu})(x)+C.$$
	Combining Lemma \ref{lem:barLf} with Proposition \ref{prop:baru-upper-bound},  for $r\gg1$, one has
	$$\left(-\frac{2}{(n-1)!|\mathbb{S}^n|}\int_{\mr^n}Q_g^{(n)}e^{nu}\ud x+o(1)\right)\log r\leq  C$$
	which implies that 
	\begin{equation*}
		\int_{\mr^n}Q_g^{(n)}e^{nu}\ud x\geq 0.
	\end{equation*}
	Thus,    one has 
	\begin{equation*}
		0\leq\frac{2}{(n-1)!|\mathbb{S}^n|}\int_{\mr^n}Q_g^{(n)}e^{nu}\ud x \leq 2.
	\end{equation*}
	
Therefore, we finish our proof.
\end{proof}
		For brevity, set the notation $\alpha_0$  throughout this paper as follows 
	\begin{equation}\label{alpha_0}
		\alpha_0:=\frac{2}{(n-1)!|\mathbb{S}^n|}\int_{\mr^n}Q_g^{(n)}e^{nu}\ud x.
	\end{equation}
	
	 Choosing $Q^{(n)}_g=(n-1)!$ in the equation \eqref{nth order Q}, one has
	\begin{equation}\label{constant Q}
		(-\Delta)^{\frac{n}{2}}u=(n-1)!e^{nu}.
	\end{equation}
For each $\lambda>0$ and $x_0\in \mr^n$, the following family solutions
\begin{equation}\label{standard solutions}
	u(x)=\log\frac{2\lambda}{\lambda^2+|x-x_0|^2}
\end{equation}
satisfies the equation \eqref{constant Q}.

For \( n = 2 \), under the finite volume assumption i.e. $e^{2u}\in L^1(\mr^2)$, Chen and Li \cite{CL} showed that all solutions to \eqref{constant Q} take the form \eqref{standard solutions}. However, for \( n \geq 4 \), the finite volume condition alone is no longer sufficient to guarantee such a conclusion. In \cite{Lin, WX}, solutions of \eqref{constant Q} were classified under the assumptions \( e^{nu} \in L^1(\mathbb{R}^n) \) and \( u(x) = o(|x|^2) \). In \cite{Chang-Yang MRL} (see line 10, page 94), Chang and Yang raised the question of whether there exist natural geometric conditions under which functions satisfying \eqref{constant Q} must be of the form \eqref{standard solutions}. Under the finite volume assumption and the condition that the scalar curvature is non-negative outside a compact set, it follows from the result of Chang, Qing, and Yang \cite{CQY} and \cite{Chang-Yang MRL, Xu xw} that every solution to \eqref{constant Q} must indeed be of the form \eqref{standard solutions}(see also \cite{Marti}).

With the aid of Theorem \ref{thm: int-Q^n-between 0,2}, we are able to remove the finite volume assumption, since it is automatically satisfied under the scalar curvature restrictions—a feature that differs from the two-dimensional case.

\begin{theorem}
	For even integer $n\geq 4$.
	If the scalar curvature of the conformal metric $g=e^{2u}|dx|^2$ on $\mr^n$ is non-negative outside a compact set, functions satisfying \eqref{constant Q} are necessarily of the form \eqref{standard solutions}.
\end{theorem}
\begin{proof}
	With the help of Theorem \ref{thm: int-Q^n-between 0,2}, one has
	$e^{nu}\in L^1(\mr^n).$ Moreover, due to Lemma \ref{lem: normal metric}, the solutions to \eqref{constant Q} satisfy the integral equation
	$$u(x)=\frac{2}{(n-1)!|\mathbb{S}^n|}\int_{\mr^n}\log\frac{|y|}{|x-y|}(n-1)!e^{nu(y)}\ud y+C.$$
	With the help of the classification theorem in \cite{Chang-Yang MRL, Xu xw} for integral equations, we finish the proof.
\end{proof}

		Now, we are going to finish this section by  providing the proof of Theorem \ref{thm: int-Q^n-between 0,1}.
	
	\hspace{3em}
	
	\noindent {\bf Proof of Theorem \ref{thm: int-Q^n-between 0,1}:}
	From Theorem \ref{thm: int-Q^n-between 0,2}, one has $Q_g^{(n)}e^{nu}\in L^1(\mathbb{R}^n)$ and $0\leq \alpha_0\leq 2.$
		Now, the statement $\alpha_0\leq 1$ follows from Theorem \ref{thm:CQY-thm} since the isoperimetric ratio is always non-negative. \qed

	\section{Positivity of $Q$-curvatures and Ricci curvature}\label{sect:positive-Q_4}
	
	From the foundational work of \cite{SY Inv}, it is known that non-negative scalar curvature imposes significant geometric and topological constraints on locally conformally flat manifolds. This observation naturally motivates the study of non-negativity conditions for other curvature quantities.
	Broadly speaking, two main directions have emerged in this line of research:
	\begin{enumerate}
		\item The study of the $kth$ elementary symmetric functions of the Schouten tensor, as explored in \cite{CHY}, \cite{G}, and \cite{GLW}.
		\item 
		Investigations concerning $Q$-curvature, as seen in \cite{Chang-Han-Yang}, \cite{CHY}, \cite{GMS}, \cite{Wang Fang}, and \cite{Zhang}.
	\end{enumerate}
	More topics related to locally conformally flat manifolds can be found in Chapter 6 in \cite{Schoen-Yau}.
	
	In this section, we focus on the second perspective.	Building on the  work \cite{LX}, we intend to show that under the assumptions of  non-negativity of $Q^{(n)}_g$ and $R_g$, the lower order $Q$-curvatures are also non-negative.  Such result will play a crucial role in the proof of gap theorem for $Q^{(2k)}_g$-curvatures.   Due to the technique limitation, we are only able to show $Q^{(4)}_g\geq 0$ currently, although we conjectured that $Q^{2k}_g\geq 0$ for all $ 1\leq k \leq \frac{n-2}{2}$.
	
	Since the fundamental work of Kuiper 
	
	\begin{theorem}\label{thm:positive Q_4}
		Let $g=e^{2u}|dx|^2$ on $\mr^n$  with $n\geq 4$  an even integer be a conformal metric. Suppose that  nth-order Q-curvature  $Q^{(n)}_g$  is  non-negative and  the   scalar curvature $R_g$ is non-negative outside a compact set.  Then,  for any integer  $1\leq k\leq \min\{2,\frac{n-2}{2}\}$, there holds
		$$Q^{(2k)}_g>0\quad \mathrm{or}\quad Q^{(2k)}_g\equiv 0.$$
	\end{theorem}
	\begin{proof}
		With the help of Theorem \ref{thm: int-Q^n-between 0,2} and Lemma \ref{lem: normal metric}, the  conformal metric $g$  is normal and has finite total $Q^{(n)}_g$-curvature.  For simplicity, denote by $\ud\nu(y)$ the non-negative  measure $$\frac{2}{(n-1)!|\mathbb{S}^n|}Q^{(n)}_g(y)e^{nu(y)}\ud y$$
		Immediately, one has  $\alpha_0=\int_{\mr^n}\ud\nu(y)$ and Theorem \ref{thm: int-Q^n-between 0,2} yields that 
		\begin{equation}\label{alpha_0 range}
			0\leq \alpha_0\leq 2.
		\end{equation}
		Since $g$ is normal, one has the following representation of $u$
		$$u(x)=\int_{\mr^n}\log\frac{|y|}{|x-y|}\ud\nu(y)+C.$$
		A direct computation yields that 
		\begin{equation}\label{equ-Deltau}
			-\Delta u(x)=\int_{\mr^n}\frac{n-2}{|x-y|^2}\ud\nu(y).
		\end{equation}
		Using Fubini's theorem and \eqref{equ-Deltau},  there holds
		\begin{align*}
			|\nabla u(x)|^2=&\int_{\mr^n}\int_{\mr^n}\frac{(x-y)\cdot(x-z)}{|x-y|^2|x-z|^2}\ud\nu(y)\ud\nu(z)\\
			=&\int_{\mr^n}\int_{\mr^n}\frac{|x-y|^2+|x-z|^2-|y-z|^2}{2|x-y|^2|x-z|^2}\ud\nu(y)\ud\nu(z)\\
			=&\frac{-\Delta u}{n-2}\alpha_0-\int_{\mr^n}\int_{\mr^n}\frac{|y-z|^2}{2|x-y|^2|x-z|^2}\ud\nu(y)\ud\nu(z).
		\end{align*}
		Denote the last term on the right by $-b(x)$ to rewrite the equation as follows:
		\begin{equation}\label{def of b}
			b(x):=\frac{-\alpha_0}{n-2}\Delta u-|\nabla u|^2=\int_{\mr^n}\int_{\mr^n}\frac{|y-z|^2}{2|x-y|^2|x-z|^2}\ud\nu(y)\ud\nu(z).
		\end{equation}
		Obviously,  due to  the assumption that $Q^{(n)}_g\geq 0$, one has $b(x)\geq 0.$

		For $k=1$ in \eqref{Q_g^2k}, one has
		\begin{align*}
			Q^{(2)}_ge^{\frac{n+2}{2}u}=&-\frac{n-2}{2}e^{\frac{n-2}{2}u}(\Delta u+\frac{n-2}{2}|\nabla u|^2)\\
			=&\frac{n-2}{2}e^{\frac{n-2}{2}u}\left(\frac{2-\alpha_0}{2}(-\Delta u)+\frac{n-2}{2}b\right).
		\end{align*}
		Combining the equations \eqref{equ-Deltau},  \eqref{def of b} and \eqref{alpha_0 range}, we easily obtain that 
		$$Q^{(2)}_g\geq 0.$$
		
		Moreover, if $Q^{(2)}_g(x_0)$ vanishes at  some point $x_0\in \mr^n$, then necessarily  $b(x_0)=0$. This in turn implies that the $nth$-order $Q^{(n)}_g$ vanishes everywhere by the definition of $b$ (see \eqref{def of b}).  Sequentially, since $g$ is normal,  it follows from the representation formula \eqref{normal-metric-def} that  $u$ must be a constant.  Finally, we obtain that $Q^{(2)}_g$ vanishes everywhere.

		Now, we are going to deal with $k=2$ if $n\geq 6$.
		
		Making use of the equations \eqref{Q_g^2k} and \eqref{def of b}, we compute directly to get the following equation:
		\begin{align*}
			\Delta^2e^{\frac{n-4}{2}u}=&\Delta\left(e^{\frac{n-4}{2}u}(\frac{n-4}{2}\Delta u+\frac{(n-4)^2}{4}|\nabla u|^2)\right)\\
			=&\frac{(n-4)^2}{4}\Delta\left(e^{\frac{n-4}{2}u}(c_1\Delta u-b)\right),
		\end{align*}
		where the constant  $c_1=\frac{2}{n-4}-\frac{\alpha_0}{n-2}.$ Moreover, $c_1>0$ due to \eqref{alpha_0 range}.
		Via the direct computation, there holds
		\begin{align*}
			\Delta^2e^{\frac{n-4}{2}u}=&\frac{(n-4)^4}{16}e^{\frac{n-4}{2}u}(c_1\Delta u-b)^2+\frac{(n-4)^2}{4}e^{\frac{n-4}{2}u}(c_1\Delta^2u-\Delta b)\\
			&+\frac{(n-4)^3}{4}e^{\frac{n-4}{2}u}\langle\nabla u, \nabla (c_1\Delta u-b)\rangle.
		\end{align*}
		
		Now, we are going to deal with these terms one by one. Firstly, the direct computation yields that 
		\begin{equation}\label{equ-Delta^2u}
			\Delta^2u=2(n-2)(n-4)\int_{\mr^n}\frac{1}{|x-y|^4}\ud\nu(y)
		\end{equation}
		which is non-negative. 
		
		Using \eqref{equ-Deltau},  \eqref{equ-Delta^2u} and Fubini's theorem, one has 
		\begin{align*}
			\langle\nabla u,\nabla \Delta u \rangle=&-2(n-2)\int_{\mr^n}\int_{\mr^n}\frac{(x-y)\cdot(x-z)}{|x-y|^2|x-z|^4}\ud\nu(y)\ud\nu(z)\\
			=&(n-2)\int_{\mr^n}\int_{\mr^n}\frac{-|x-y|^2-|x-z|^2+|y-z|^2}{|x-y|^2|x-z|^4}\ud\nu(y)\ud\nu(z)\\
			=&-\frac{\alpha_0}{2(n-4)}\Delta^2u-\frac{(\Delta u)^2}{n-2}+(n-2)\int_{\mr^n}\int_{\mr^n}\frac{|y-z|^2}{|x-y|^2|x-z|^4}\ud\nu(y)\ud\nu(z).
		\end{align*}
		
		For convenience,  we define the function $h(x)$ by the following equation:
		\begin{equation}\label{def of h}
			h(x):=\int_{\mr^n}\int_{\mr^n}\frac{|y-z|^2}{|x-y|^2|x-z|^4}\ud\nu(y)\ud\nu(z)\geq 0.
		\end{equation}
		
		Thus, one has
		\begin{equation}\label{nabla u,nabla Delta u}
			\langle\nabla u,\nabla \Delta u \rangle=-\frac{\alpha_0}{2(n-4)}\Delta^2u-\frac{(\Delta u)^2}{n-2}+(n-2)h.
		\end{equation}
		Via the direct computation and  the symmetry,  we obtain the estimate:
		\begin{align*}
			\langle\nabla u,\nabla b\rangle=&\int_{\mr^n}\int_{\mr^n}\int_{\mr^n}\frac{[(x-y)\cdot(x-z)]|z-w|^2}{|x-y|^2|x-z|^4|x-w|^2}\ud\nu(y)\ud\nu(z)\ud\nu(w)\\
			&+\int_{\mr^n}\int_{\mr^n}\int_{\mr^n}\frac{[(x-y)\cdot(x-w)]|z-w|^2}{|x-y|^2|x-z|^2|x-w|^4}\ud\nu(y)\ud\nu(z)\ud\nu(w)\\
			=&2\int_{\mr^n}\int_{\mr^n}\int_{\mr^n}\frac{[(x-y)\cdot(x-z)]|z-w|^2}{|x-y|^2|x-z|^4|x-w|^2}\ud\nu(y)\ud\nu(z)\ud\nu(w)\\
			= &\int_{\mr^n}\int_{\mr^n}\int_{\mr^n}\frac{(|x-y|^2+|x-z|^2-|y-z|^2)|z-w|^2}{|x-y|^2\cdot|x-z|^4|x-w|^2}\ud\nu(y)\ud\nu(z)\ud\nu(w)\\
			\leq &\int_{\mr^n}\int_{\mr^n}\int_{\mr^n}\left(\frac{1}{|x-y|^2|x-z|^2}+\frac{1}{|x-z|^4}\right)\frac{|z-w|^2}{|x-w|^2}\ud\nu(y)\ud\nu(z)\ud\nu(w).
		\end{align*}
		Hence, with the help of the representations \eqref{equ-Deltau},  \eqref{def of b} and \eqref{def of h}, one gets
		\begin{equation}\label{nabla u,nabla b geq Delta ub +h}
			\langle\nabla u,\nabla b\rangle\leq \frac{2}{n-2}(-\Delta u)b+h.
		\end{equation}
		Similarly, one can also obtain the following identity:
		\begin{align*}
			\Delta b=&-(n-4)\int_{\mr^n}\int_{\mr^n}\frac{|y-z|^2}{|x-y|^4|x-z|^2}\ud\nu(y)\ud\nu(z)\\
			&-(n-4)\int_{\mr^n}\int_{\mr^n}\frac{|y-z|^2}{|x-y|^2|x-z|^4}\ud\nu(y)\ud\nu(z)\\
			&+4\int_{\mr^n}\int_{\mr^n}\frac{[(x-y)\cdot(x-z)]|y-z|^2}{|x-y|^4|x-z|^4}\ud\nu(y)\ud\nu(z)\\
			=&-2(n-4)\int_{\mr^n}\int_{\mr^n}\frac{|y-z|^2}{|x-y|^2|x-z|^4}\ud\nu(y)\ud\nu(z)\\
			&+2\int_{\mr^n}\int_{\mr^n}\frac{(|x-y|^2+|x-z|^2-|y-z|^2)|y-z|^2}{|x-y|^4|x-z|^4}\ud\nu(y)\ud\nu(z)\\
			=&-2(n-6)\int_{\mr^n}\int_{\mr^n}\frac{|y-z|^2}{|x-y|^2|x-z|^4}\ud\nu(y)\ud\nu(z)\\
			&-2\int_{\mr^n}\int_{\mr^n}\frac{|y-z|^4}{|x-y|^4|x-z|^4}\ud\nu(y)\ud\nu(z).
		\end{align*}
		With the help of above identity, one will have no difficulty to reach:
		\begin{equation}\label{equ-Deltab lower bound}
			-\Delta b\geq 2(n-6)h.
		\end{equation}

		The estimates \eqref{nabla u,nabla Delta u}, \eqref{nabla u,nabla b geq Delta ub +h} and \eqref{equ-Deltab lower bound} can be used to get the first estimate:
		\begin{align*}
			&\frac{4}{(n-4)^2}e^{-\frac{n-4}{2}u}\Delta^2e^{\frac{n-4}{2}u}\\
			\geq &\frac{(n-4)^2}{4}\left(c_1^2(\Delta u)^2+2c_1(-\Delta u)b+b^2\right)+c_1\Delta^2u+2(n-6)h\\
			&-\frac{\alpha_0c_1}{2}\Delta^2u-\frac{(n-4)c_1}{n-2}(\Delta u)^2+(n-2)(n-4)c_1h\\
			&-\frac{2(n-4)}{n-2}(-\Delta u)b-(n-4)h\\
			=&\frac{(n-4)^2(2-\alpha_0)c_1}{4(n-2)}(\Delta u)^2+\frac{(2-\alpha_0)c_1}{2}\Delta^2u\\
			&+\frac{(n-4)^2(2-\alpha_0)}{2(n-2)}(-\Delta u)b+(n-4)(3-\alpha_0)h.
		\end{align*}
		Next, keep in mind that  $(-\Delta u), b$ and  $h$ are all non-negative  and the non-negative number $\alpha_0$ is no greater than $2$ \eqref{alpha_0 range} to easily obtain: 
		\begin{equation}\label{equ-Delta^e^n-4/2u}
			\frac{4}{(n-4)^2}e^{-\frac{n-4}{2}u}\Delta^2e^{\frac{n-4}{2}u}\geq \frac{(n-4)^2(2-\alpha_0)c_1}{4(n-2)}(\Delta u)^2+\frac{(2-\alpha_0)c_1}{2}\Delta^2u.
		\end{equation}
		By the definition  \eqref{Q_g^2k} for $Q^{(4)}_g$ and the constant $c_1$ \eqref{equ-Delta^e^n-4/2u}, one should rewrite above estimate as:
		\begin{equation}\label{positive Q^4}
			\frac{4}{(n-4)^2}Q^{(4)}_ge^{4u}\geq \left(2-\alpha_0\right)\left(\frac{2(n-2)}{n-4}-\alpha_0\right)\left(\frac{(n-4)^2(\Delta u)^2}{4(n-2)^2}+\frac{\Delta^2 u}{2(n-2)}\right).
		\end{equation}
		As was pointed out early that $\Delta^2 u \geq 0$ and $0 \leq \alpha_0 \leq 2$, we can rather easily see that  
		$$Q^{(4)}_g \geq 0.$$
		Due to the same reason as before, if $Q^{(4)}_g$ vanishes at some point, it must vanish everywhere.
		
		The proof is complete.
	\end{proof}

	Now, we are going to show that the positivity of $nth$-order $Q$-curvature implies the positivity  of Ricci curvature.
	
	\begin{theorem}\label{thm:positive Ricci for normal metric}
		For a normal metric $g=e^{2u}|dx|^2$   on $\mr^n$ where 
		$n\geq 4$ is an even integer, if the $nth$-order $Q$-curvature is non-negative and its total integral satisfies the estimate
		$$\int_{\mr^n}Q^{(n)}_g\ud\mu_g\leq (n-1)!|\mathbb{S}^n|,$$
		then its Ricci curvature must be non-negative.
	\end{theorem}
	\begin{proof}
		
		It is well known that, Ricci curvature tensor of the conformal metric $g=e^{2u}|dx|^2$ can be written as (see Besse's classical book \cite{B})
		\begin{equation}\label{Ricci-equ}
			(Ric_g)_{ij}=-(n-2)u_{ij}+(n-2)u_iu_j-\Delta u\delta_{ij}-(n-2)|\nabla u|^2\delta_{ij}.
		\end{equation}
		
		Since $g$ is normal, up to adding a constant, $u$ is the logarithmic potential of $Q^{(n)}_g e^{nu}$. One continues to use the notation $d\nu$ to indicate the measure element as before. Thus one obtains:
		$$u_{i}=-\int_{\mr^n}\frac{(x_i-y_i)}{|x-y|^2}\ud\nu(y)$$
		and 
		\begin{equation}\label{u_ij-repren}
			u_{ij}=-\int_{\mr^n}\frac{\delta_{ij}}{|x-y|^2}\ud\nu(y)+2\int_{\mr^n}\frac{(x_i-y_i)(x_j-y_j)}{|x-y|^4}\ud\nu(y).
		\end{equation}
		Recall the expression of $-\Delta u$ \eqref{equ-Deltau}. Make use of above equation for $u_{ij}$ \eqref{u_ij-repren}, together with the expression for $(-\Delta u)$, for any  vector $\vec{a}=(a_1,\cdots, a_n)$, to conclude that 
		\begin{align*}
			&\sum^n_{i=1}\sum^n_{j=1}-(n-2)u_{ij}a_ia_j-\Delta u|\vec{a}|^2\\
			=&2(n-2)\int_{\mr^n}\frac{|\vec{a}|^2}{|x-y|^2}\ud\nu(y)-2(n-2)\int_{\mr^n}\frac{\sum^n_{i=1}\sum^n_{j=1}(x_i-y_i)a_i(x_j-y_j)a_j}{|x-y|^4}\ud\nu(y)\\
			=&2(n-2)\int_{\mr^n}\frac{|\vec{a}|^2\cdot|x-y|^2- (\langle x-y, \vec{a}\rangle)^2}{|x-y|^4}\ud\nu(y)
		\end{align*}
		which deduces that 
		\begin{equation}\label{u_ija_ia_jgeq 0}
			\sum^n_{i=1}\sum^n_{j=1}-(n-2)u_{ij}a_ia_j-\Delta u|\vec{a}|^2\geq 0.
		\end{equation}
		Now, we use the elementary identity: $2(|\vec{a}|^2|x-y|^2- (\langle x-y, \vec{a}\rangle)^2) = \sum^n_{i=1}\sum^n_{j=1} \left((x_i-y_i)a_j-(x_j-y_j)a_i\right)^2$ and H\"older's inequality, to obtain the estimate:
		\begin{align*}
			&\alpha_0\sum^n_{i=1}\sum^n_{j=1}\int_{\mr^n}\frac{\left((x_i-y_i)a_j-(x_j-y_j)a_i\right)^2}{|x-y|^4}\ud\nu(y)\\
			=&\sum^n_{i=1}\sum^n_{j=1}\int_{\mr^n}\ud\nu(y)\int_{\mr^n}\frac{\left((x_i-y_i)a_j-(x_j-y_j)a_i\right)^2}{|x-y|^4}\ud\nu(y)\\  \geq&\sum^n_{i=1}\sum^n_{j=1}\left(\int_{\mr^n}\frac{(x_i-y_i)a_j-(x_j-y_j)a_i}{|x-y|^2}\ud\nu(y)\right)^2\\
			=&\sum^n_{i=1}\sum^n_{j=1}(u_ia_j-a_iu_j)^2\\
			=&2\left(|\nabla u|^2\cdot |\vec{a}|^2-\sum^n_{i=1}\sum^n_{j=1}u_iu_ja_ia_j\right)
		\end{align*}
		which deduces that 
		\begin{equation}\label{u_ija_ia_jgeqnabla u^2}
			\alpha_0\left(-(n-2)u_{ij}a_ia_j-\Delta u|\vec{a}|^2\right)\geq 2(n-2)\left(|\nabla u|^2\cdot |\vec{a}|^2-u_iu_ja_ia_j\right).
		\end{equation}
		Using the  estimates \eqref{u_ija_ia_jgeq 0}, \eqref{u_ija_ia_jgeqnabla u^2} and the assumption $0\leq \alpha_0\leq 2$, one has
		\begin{align*}
			&Ric_g(\vec{a},\vec{a})\\
			=&(Ric_g)_{ij}a_ia_j\\
			=&-(n-2)u_{ij}a_ia_j-\Delta u|\vec{a}|^2+(n-2)\left(u_iu_ja_ia_j-|\nabla u|^2|\vec{a}|^2\right)\\
			\geq &(1-\frac{\alpha_0}{2})\left(-(n-2)u_{ij}a_ia_j-\Delta u|\vec{a}|^2\right)\geq 0.
		\end{align*}
		Thus, we finish the proof. 
	\end{proof}
	
	Combining the above theorem with Theorem \ref{thm: int-Q^n-between 0,2}, we obtain the following result.
	\begin{corollary}
		Let $g=e^{2u}|dx|^2$ be a conformal metric on $\mr^n$  where $n\geq 4$ is an even integer.  Suppose  the  $nth$-order $Q^{(n)}_g$-curvature  and scalar curvature are non-negative. Then, the Ricci curvature $Ric_g$ is non-negative.
	\end{corollary}
	
	Hence, Theorem \ref{thm: positive Ricci} follows from the above corollary directly.

	If we focus our attention on the conformal class of standard sphere, we may also establish such result by requiring only the positivity of $nth$-order $Q$-curvature with the help of Chern-Gauss-Bonnet formula.
	\begin{theorem}\label{thm:positive Ricci for Sn}
		Consider the  conformal metric $g=e^{2w}g_0$ on standard sphere $(\ms^n, g_0)$ where $n\geq 4$ is an even integer. If the nth-order Q-curvature $Q^{(n)}_g$ is non-negative, then the Ricci curvature  $Ric_g$ is non-negative.
	\end{theorem}
	\begin{proof}
		Via the  stereographic projection, we can identify the conformal metric $g$ on $\mathbb{S}^n$ with a metric on $\mathbb{R}^n$. Specifically, for any point $(\xi_1, \dots, \xi_{n+1}) \in \mathbb{S}^n \subset \mathbb{R}^{n+1}$ and its corresponding point $x = (x_1, \dots, x_n) \in \mathbb{R}^n$ under the projection, we have the coordinate relations:
		$$
		\xi_i = \frac{2x_i}{1+|x|^2} \quad \text{for } 1 \leq i \leq n, \quad \text{and} \quad \xi_{n+1} = \frac{1-|x|^2}{1+|x|^2}.
		$$
		Define the transformed function
		\begin{equation}\label{w_1 def}
			w_1(x) := \log\left(\frac{2}{1+|x|^2}\right) + w(\xi(x)).
		\end{equation}
		By the conformal invariance of the GJMS operator, we obtain
		$$
		(-\Delta)^{\frac{n}{2}} w_1 = Q_g^{(n)} e^{n w_1}.
		$$
		Furthermore, the Chern-Gauss-Bonnet formula on $\mathbb{S}^n$ implies that $Q_g^{(n)} e^{n w_1} \in L^1(\mathbb{R}^n)$ with 
		\begin{equation}\label{CGB formula}
			\int_{\mathbb{R}^n} Q_g^{(n)} e^{n w_1} \, dx = (n-1)!|\mathbb{S}^n|.
		\end{equation}
		For more details, we refer to \cite{Chang-Yang MRL}.
		From \eqref{w_1 def}, we easily deduce that $|w_1(x)| \leq C \log(|x|+2)$, which shows that $w_1$ is  normal. An application of Theorem \ref{thm:positive Ricci for normal metric} completes the proof.
	\end{proof}

	In the famous work \cite{Guan-V-W}, they showed that when $kth$ elementary symmetric function $\sigma_i(A_g)\geq 0$ of Schouten tensor $A_g$ for $1\leq i\leq k$ and $k\geq 2$, then Ricci curvature has a lower bound depending on scalar curvature. In our setting, we find that such phenomenon also  holds for  $Q_{g}^{(2k)}$-curvature.
	
	Same  as in \cite{LX}, we say that  the  $Q^{(2k)}_g$-curvature has \emph{slow decay barrier} near infinity if, for $|x|\gg1$ and $-2k<s\leq 0$, there holds
	$$Q^{(2k)}_g\geq C|x|^{s}.$$

	\begin{theorem}\label{thm:positive-Q-2k-implies-Ricci-lower}
		Given integers $n>2k$ and $k\geq2$, consider a conformal metric $g=e^{2u}|dx|^2$ on $\mr^n$. Suppose that  $Q^{(2k)}_g$-curvature is non-negative and has slow decay barrier near infinity. Then there holds
		$$Ric_g\geq \frac{2k-n}{4(k-1)}R_gg.$$
	\end{theorem}
	\begin{proof}
		Firstly, set the notation $v:=e^{\frac{n-2k}{2}u}$. Using \eqref{Q_g^2k}, one has
		$$(-\Delta)^kv=Q^{(2k)}_gv^{\frac{n+2k}{n-2k}}.$$
		Based on our assumptions, with the help of Theorem 2.6 in \cite{LX}, one has
		\begin{equation}\label{equ:integeral-for-v}
			v(x)=\frac{1}{C(n,k)}\int_{\mr^n}\frac{Q^{(2k)}_g(y)v(y)^{\frac{n+2k}{n-2k}}}{|x-y|^{n-2k}}\ud y
		\end{equation}
		and for any $1\leq i\leq 2k-1$, Lemma 2.4 in \cite{LX} yields that 
		\begin{equation}\label{equ:nabla-for-v}
			\nabla^i v=\frac{1}{C(n,k)}\int_{\mr^n}\nabla^i_x|x-y|^{2k-n}Q^{(2k)}_g(y)v(y)^{\frac{n+2k}{n-2k}}\ud y
		\end{equation}
		where $C(n,k)$ is a positive constant depending on $n,k$.
		For simplicity, we set the notation
		$$\ud\mu(y):=\frac{1}{C(n,k)}Q^{(2k)}_g(y)v(y)^{\frac{n+2k}{n-2k}}\ud y.$$
		A direct computation yields that 
		\begin{equation}\label{nabla-u}
			u_{i}=\frac{2}{n-2k}\frac{v_i}{v},  
		\end{equation}
		\begin{equation}\label{u_ij}
			u_{ij}=\frac{2}{n-2k}\left(\frac{v_{ij}}{v}-\frac{v_iv_j}{v^2}\right),
		\end{equation}
		
		and
		\begin{equation}\label{Delta-u}
			\Delta u=\frac{2}{n-2k}\left(\frac{\Delta v}{v}-\frac{|\nabla v|^2}{v^2}\right).   
		\end{equation}
		Thus, combining \eqref{nabla-u} with \eqref{u_ij}, one has
		\begin{equation}\label{equ:-u_ij+u_iu_j}
			-u_{ij}+u_{i}u_j=\frac{2}{n-2k}\left(-\frac{v_{ij}}{v}+\frac{n-2k+2}{n-2k}\frac{v_iv_j}{v^2}\right)
		\end{equation}
		From \eqref{equ:integeral-for-v} and \eqref{equ:nabla-for-v}, one has
		$$v_{i}=(2k-n)\int_{\mathbb{R}^n}\frac{(x_i-y_i)}{|x-y|^{n-2k+2}}\ud\mu(y)$$
		and then 
		$$v_{ij}=(2k-n)\int_{\mathbb{R}^n}\frac{\delta_{ij}}{|x-y|^{n-2k+2}}\ud\mu(y)+(2k-n)(2k-2-n)\int_{\mathbb{R}^n}\frac{(x_i-y_i)(x_j-y_j)}{|x-y|^{n-2k+4}} \ud\mu(y).$$
		Immediately, there holds 
		$$\Delta v=(2k-n)(2k-2)\int_{\mathbb{R}^n}\frac{1}{|x-y|^{n-2k+2}}\ud\mu(y).$$
		Using the above integral  identities, H\"older's inequality and the same trick during the proof of Theorem \ref{thm:positive Ricci for normal metric}, for any vector $\vec{a}=(a_1,\cdots,a_n)$, one has
		\begin{align*}
			&-vv_{ij}a_ia_j\\
			=&-\frac{v\Delta v}{2k-2}|\vec{a}|^2-(n-2k)(n+2-2k)v\int_{\mathbb{R}^n}\frac{(x_i-y_i)a_i(x_j-y_j)a_j}{|x-y|^{n-2k+4}}\ud\mu(y)\\
			=&\frac{n+1-2k}{2k-2} v\Delta v|\vec{a}|^2+(n-2k)(n+2-2k)v\int_{\mathbb{R}^n}\frac{|a|^2|x-y|^2-(x_i-y_i)a_i(x_j-y_j)a_j}{|x-y|^{n-2k+4}}\ud\mu(y)\\
			=&\frac{n+1-2k}{2k-2}|\vec{a}|^2v\Delta v\\
			&+\frac{(n-2k)(n+2-2k)}{2}\int_{\mathbb{R}^n}\frac{(a_j(x_i-y_i)-a_i(x_j-y_j))^2}{|x-y|^{n-2k+4}}\ud\mu(y)\int_{\mathbb{R}^n}\frac{1}{|x-y|^{n-2m}}\ud\mu(y)\\
			\geq &\frac{n+1-2k}{2k-2}|\vec{a}|^2 v\Delta v+\frac{(n-2k)(n+2-2k)}{2}\left(\int_{\mathbb{R}^n}\frac{a_j(x_i-y_i)-a_i(x_j-y_j)}{|x-y|^{n-2k+2}}\ud\mu(y)\right)^2\\
			=&\frac{n+1-2k}{2k-2}|\vec{a}|^2v\Delta v+\frac{n+2-2k}{n-2k}(|\vec{a}|^2|\nabla v|^2-a_ia_jv_iv_j).
		\end{align*}
		Hence, using \eqref{equ:-u_ij+u_iu_j} and the above estimate, one has
		\begin{equation}\label{-u_ij+u_iu_j}
			(n-2)(-u_{ij}a_ia_j+u_{i}u_ja_ia_j)\geq \frac{2(n-2)}{(n-2k)v^2}\left(\frac{n+1-2k}{2k-2}|\vec{a}|^2v\Delta v+\frac{n+2-2k}{n-2k}|\vec{a}|^2|\nabla v|^2\right)
		\end{equation}
		Thus, inserting  \eqref{Delta-u}, \eqref{nabla-u} and \eqref{-u_ij+u_iu_j} into \eqref{Ricci-equ}, we obtain that 
		$$(Ric_g)_{ij}a_ia_j\geq \frac{n-1}{k-1}\frac{\Delta v}{v}|\vec{a}|^2+\frac{2(n-1)}{n-2k}\frac{|\nabla v|^2}{v^2}|\vec{a}|^2$$
		Notice that the scalar curvature $R_g$ satisfies the following equation
		$$R_ge^{2u}=2(n-1)(-\Delta u-\frac{n-2}{2}|\nabla u|^2)=\frac{4(n-1)}{n-2k}\left(-\frac{\Delta v}{v}+\frac{2-2k}{n-2k}\frac{|\nabla v|^2}{v^2}\right)$$
		and $R_gg(\vec{a},\vec{a})=R_ge^{2u}|\vec{a}|^2$.
		Since the vector $\vec{a}$ is arbitrary, 
		there holds
		$$Ric_g\geq \frac{2k-n}{4(k-1)}R_gg.$$
		
		Thus, we finish our proof.
	\end{proof}
	
	Indeed, the slow decay barrier condition is imposed solely to guarantee that the conformal factor satisfies the integral representations \eqref{equ:integeral-for-v} and \eqref{equ:nabla-for-v}. When examining conformal metrics on the standard sphere through stereographic projection, one can readily verify that conditions \eqref{equ:integeral-for-v} and \eqref{equ:nabla-for-v} remain valid. Consequently, we obtain the following corollary to Theorem \ref{thm:positive-Q-2k-implies-Ricci-lower}
	\begin{corollary}
		Given integers $n>2k$ and $k\geq 2$. Consider the conformal metric $g=e^{2w}g_0$ on standard sphere $(\ms^n,g_0)$. Suppose that $2kth$-order $Q$-curvature $Q^{(2k)}_g$ is non-negative. Thus, there holds
		$$Ric_g\geq \frac{2k-n}{4(k-1)}R_gg.$$
	\end{corollary}

	\section{Comparison of metric distance and measure distance}\label{sec:A_inftyweight}
	
	Another ingredient for our argument is the concept: strong $A_\infty$ weight and its properties. For comprehensive treatments, readers  are refered to articles \cite{DS} and \cite{Semmes}.
	
	Our presentation here starts with the definition. Let $\mu = w(x)dx$ be a measure on $\mathbb{R}^n$ and suppose that $w$ is a positive locally integrable function. The function $w$ is said to be an \emph{$A_\infty$ weight} if there exists a constant $C > 0$ such that for every ball $B \subset \mathbb{R}^n$,
	$$
	\fint_{B} w\,dx \leq C \exp\left(\fint_B \log w\,dx\right).
	$$
	
	If the function $w$ is strictly positive everywhere on $\mathbb{R}^n$, we may write it as $e^{nu}$. Associated it, there will be a conformal metric $g = e^{2u}|dx|^2$ on $\mathbb{R}^n$. With this conformal metric, we can consider its distance geometry. Due to this consideration, Semmes \cite{Semmes} establishs the following fundamental inequality
	\begin{equation}\label{d_g leq delta_g}
		d_g(x,y) \leq C\delta_g(x,y),
	\end{equation}
	where: $d_g(\cdot,\cdot)$  is the geodesic distance of $g$ and  $\delta_g(x,y)$ is the measure distance defined by
	\[
	\delta_g(x,y) := \left(\int_{B_{x,y}} e^{nu}\,dz\right)^{\frac{1}{n}}
	\]
	with  $B_{x,y} = B_{|x-y|/2}\left(\frac{x+y}{2}\right)$ the smallest Euclidean ball containing $x$ and $y$.
	
	Naturally this inequality leads us to ask  when those distances are equivalent. Clearly this is not always true.  David and Semmes \cite{DS} introduced  a \emph{strong $A_\infty$ weight} if they are equivalent, that is, the reverse inequality
	\begin{equation}\label{delta_g leq Cd_g}
		\delta_g(x,y) \leq Cd_g(x,y),
	\end{equation}
	also holds true for some constant $C > 0$.
	
	It is well known that the doubling property of the volume measure is important to do analysis on the manifolds.  This crucial property for $A_\infty$ weights has been found by Semmes \cite{Semmes}: that is, 
	\begin{equation}\label{doubling property}
		\int_{B_{2r}(x_0)} e^{nu}\,dx \leq C \int_{B_r(x_0)} e^{nu}\,dx,
	\end{equation}
	where $C > 0$ is independent of $x_0 \in \mathbb{R}^n$ and $r > 0$.

	Get back to  our  normal metric $g=e^{2u}|dx|^2$ on $\mr^n$. Bonk, Heinonen, and Saksman  \cite{B-H-S} use  quasi-conformal flows to prove  that $e^{nu}$ is a strong $A_\infty$ weight if its total $Q^{(n)}_g$-curvature is relatively small. A few years later, Wang \cite{Wang IMRN} extends their result to the optimal assumption on $Q^{(n)}_g$-curvature integral (See also  \cite{Sire-Wang}, \cite{Wang Yi} for more related topics):
	
	\begin{theorem}\label{thm:Wang's theorem}(See Corollary 1.7 in \cite{Wang IMRN})
		Let $g=e^{2u}|dx|^2$ on $\mr^n$  with $n\geq 4$ an even integer be a smooth conformal metric. 
		Suppose that the metric $g$ is  normal and its $Q^{(n)}_g$ curvature satisfies
		$$\int_{\mr^n}Q^{(n)}_ge^{nu}\ud x<\frac{(n-1)!|\mathbb{S}^n|}{2}.$$ Then $e^{nu}$ is a strong $A_\infty$ weight.
	\end{theorem}
	
	As a consequence of this theorem, we observe that  $e^{n\bar{u}}$ is also a strong $A_\infty$ weight. 
	\begin{lemma}\label{lem:strong Ainfty for bar u}
		Under the same assumption as in Theorem  \ref{thm:Wang's theorem}. Then, $e^{n\bar u}$ is a strong $A_\infty$ weight.
	\end{lemma}	
	
	\begin{proof}
		Consider a new conformal metric $\bar g=e^{2\bar u}|d x|^2$. First we notice that 
		\begin{equation}\label{bar -Delta^n/2u}
			(-\Delta)^{\frac{n}{2}}\bar u(r)=\fint_{\partial B_r(0)}(-\Delta)^{\frac{n}{2}}u\ud\sigma=\fint_{\partial B_r(0)}Q_g^{(n)}e^{nu}\ud x.
		\end{equation}
		Thus,  there holds
		\begin{align*}
			\int_{\mr^n}|(-\Delta)^{\frac{n}{2}}\bar u(|x|)|\ud x\leq & \int_0^\infty |\partial B_r(0)|\cdot|(-\Delta)^{\frac{n}{2}}\bar u(r)|\ud r\\
			=&\int_0^\infty\left|\int_{\partial B_r(0)}Q_g^{(n)}e^{nu}\ud\sigma \right|\ud r\\
			\leq  &\int_{\mr^n}|Q_g^{(n)}|e^{nu}\ud x<+\infty.
		\end{align*}
		Thus $\bar g$ also has finite total $Q^{(n)}_g$-curvature. Since $g$ is normal, Lemma \ref{lem:barLf} yields that , for $r$ sufficiently large,
		$$\bar u(r)=(-\alpha_0+o(1))\log r$$
		where $\alpha_0$ is same as that in \eqref{alpha_0}.
		Using the above  estimate, one has , for $r$ large,
		$$\fint_{B_r(0)}|\bar u(|x|)|\ud x\leq C\log r,  \quad r\gg1. $$
		This implies that $\bar g$ is also a normal metric by using Lemma \ref{lem:barLf} and Corollary \ref{cor:Liouville}.
		Using \eqref{bar -Delta^n/2u}, one has
		\begin{align*}
			\int_{\mr^n}Q_{\bar g}^{(n)}e^{n\bar u}\ud x=&\int_{\mr^n}(-\Delta)^{\frac{n}{2}}\bar u\ud x\\
			=&\int^\infty_0\int_{\partial B_r(0)}Q^{(n)}_ge^{nu}\ud\sigma\ud r\\
			=&\int_{\mr^n}Q_g^{(n)}e^{nu}\ud x<\frac{(n-1)!|\mathbb{S}^n|}{2}
		\end{align*}
		Then, applying Theorem \ref{thm:Wang's theorem}, we show that $e^{n\bar u}$ is also a strong $A_\infty$ weight.
	\end{proof}

	Using the properties of strong $A_\infty$ weight, a distance comparison identity is established in \cite{Li 23 Q-curvature} which plays an important role during the proof of gap theorems for $Q^{(2k)}_g$-curvatures later.
	
	\begin{theorem}\label{thm:length identity}(Theorem 1.4 in \cite{Li 23 Q-curvature})
		Given a conformal metric $g=e^{2u}|dx|^2$ on $\mr^n$  where $n\geq 4$ is an even integer with finite total $Q^{(n)}_g$-curvature. 
		Supposing that $g$ is a normal metric, then for any fixed point $p_0\in\mr^n$, there holds
		\begin{equation}\label{distance comparison identity}
			\lim_{|x| \to \infty}\frac{\log d_g(x,p_0)}{\log |x-p_0|}=\max\{1-\alpha_0, 0\}.
		\end{equation}
	\end{theorem}

	\section{Average growth  rate of $\sigma_k$ and  $Q^{(2k)}_g$}\label{sec: proof of integral growth}
	
	Let us start with a calculus lemma:
	\begin{lemma}\label{lem:nabal^kvarphi}
		For any $\varphi\in C^{k}(\mr^n)$ and integer $k\geq 1$ , there holds
		$$|\nabla^k e^{\varphi}|\leq C(n,k)e^{\varphi}\sum^{k}_{i=1}|\nabla^i \varphi|^{\frac{k}{i}}$$
		where $C(n,k)$ is a positive constant depending on $k,n$.
	\end{lemma}
	
	\begin{proof}
		Via a direct computation, one has
		$$\nabla^ke^{\varphi}=e^{\varphi} \sum_{i=1}^k \sum_{\substack{{a_{i,1}+\cdots+ a_{i,i}}=k\\  a_{i,j}\in\{1, 2, \cdots, k\}}} C(a_{i,1},\cdots, a_{i,i}) \nabla^{a_{i,1}}\varphi \cdots \nabla^{a_{i,i}}\varphi$$
		where $C(a_{i,1},\cdots, a_{i,i})$ are some constants. Then, applying Young's inequality, one has
		$$|\nabla^ke^{\varphi}|\leq C(n,k)e^{\varphi}\sum_{i=1}^k\sum_{\substack{{a_{i,1}+\cdots+ a_{i,i}}=k\\  a_{i,i}\in \{1, 2, \cdots, k\}}} \left(|\nabla^{a_{i,1}}\varphi|^{\frac{k}{a_{i,1}}}+\cdots+|\nabla^{a_{i,i}}\varphi|^{\frac{k}{a_{i,i}}}\right).$$
		Reorganize the terms to get:
		$$|\nabla^k e^{\varphi}|\leq C(n,k)e^{\varphi}\sum^{k}_{i=1}|\nabla^i \varphi|^{\frac{k}{i}}.$$
	\end{proof}
	
	\begin{lemma}\label{lem:int_B_r Q^2k)p}
		Let  a normal metric $g=e^{2u}|dx|^2$ on $\mr^n$  with $n\geq 4$ an even integer be given. For each natural number $k \leq \frac{n-2}{2}$ and each  real number $p$ with restriction $1\leq p<\frac{n-1}{2k}$, there exists a positive constant $ C$ such that the following estimate holds: for $r \gg 1$,
		$$\int_{B_r(0)}|Q_g^{(2k)}|^pe^{nu}\ud x\leq C+C\int^r_1s^{n-2kp-1}e^{(n-2kp)\bar u(s)}\ud s.$$
	\end{lemma}
	\begin{proof}
		Make use  Lemma \ref{lem:nabal^kvarphi} to get the estimate:
		\begin{align*}
			|Q^{(2k)}_g|^p=&|e^{-\frac{n+2k}{2}u}(-\Delta)^{k}e^{\frac{n-2k}{2}u}|^p\\
			\leq &C\left(e^{-2ku}\sum_{i=1}^{2k}|\nabla^i u|^{\frac{2k}{i}}\right)^p\\
			\leq &Ce^{-2kpu}\sum_{i=1}^{2k}|\nabla^i u|^{\frac{2kp}{i}}.
		\end{align*}
		Since $g$ is normal, combine this estimate with  Lemma \ref{lem:equ nabla^k Lf} to deduce that
		\begin{align*}
			&\int_{B_r(0)}|Q^{(2k)}_g|^pe^{nu}\ud x\\
			\leq &C+ \int^r_1|\partial B_s(0)|\fint_{\partial B_s(0)}|Q^{(2k)}_g|^pe^{nu}\ud\sigma\ud s\\
			\leq &C+C \int^r_1s^{n-1}\fint_{\partial B_s(0)}e^{(n-2kp)u}\sum_{i=1}^{2k}|\nabla^i u|^{\frac{2kp}{i}}\ud\sigma\ud s\\
			\leq& C+C \int^r_1s^{n-1-2kp}e^{(n-2kp)\bar u(s)}\ud s.
		\end{align*}
		Thus, our proof is complete.
	\end{proof}
	
	Over a distance ball regarding to background metric, we can derive the growth for $2k th$-order $Q$-curvature.
	\begin{theorem}\label{thm: Q_2k growth on Euclidean-for-L^p}
		Let a complete and  conformal metric $g=e^{2u}|dx|^2$ on $\mr^n$  with the even dimension at least $4$ be given. Suppose its scalar curvature $R_g\geq 0$ outside a compact set and the negative part of $nth$-order Q-curvature satisfies $(Q^{(n)}_g)^-e^{nu}\in L^1(\mr^n)$.   Then,   for each fixed point $p_0$,  each integer  $ 1\leq k\leq \frac{n-2}{2}$ and each real number $1\leq p<\frac{n-1}{2k}$,  there holds
		$$
		\limsup_{r\to\infty}r^{2kp-n}\int_{B_r(p_0)}|Q_g^{(2k)}|^p\ud \mu_g<+\infty
		$$
		where $B_r(p_0)$ is the Euclidean ball centered at point $p_0$  with radius  $r$. 
	\end{theorem}
	\begin{proof}
		First of all, Theorem \ref{thm: int-Q^n-between 0,1} shows that  $Q^{(n)}_ge^{nu}\in L^1(\mr^n)$. Thus Lemma \ref{lem: normal metric} implies that  $g$ is a normal metric.  Now, applying Lemma \ref{lem:int_B_r Q^2k)p} and Proposition \ref{prop:baru-upper-bound},   we obtain, for $r\gg1$,
		\begin{align*}
			\int_{B_r(0)}|Q^{(2k)}_g|^pe^{nu}\ud x\leq &C+C\int^r_1s^{n-2kp-1}e^{(n-2kp)\bar u(s)}\ud s\\
			\leq &C+C\int^r_1s^{n-2kp-1}\ud s\\
			\leq &Cr^{n-2kp}.
		\end{align*}
		This completes the proof.
	\end{proof}

	This result roughly says that the average of $2kth$-order $Q$-curvature over the Euclidean ball decays at the order at least $2k$ in $L^p$. This is almost what we have claimed, the only difference is that the ball here is not the geodesic ball. Our claim is due to the observation  we presented in Theorem \ref{thm:length identity}. Now, we are ready to show our average growth rate of $Q^{(2k)}_g$-curvature over geodesic balls. Since our approach will rely on the properties of strong $A_\infty$ weights, we necessarily assume that $\alpha_0 < 1$.

	\begin{theorem}\label{thm:Q^2k_g L^p growth for normal metric}
		Once again let us consider a conformal complete normal metric $g=e^{2u}|dx|^2$ on $\mr^n$  with even dimension at least $4$.  Assume $\alpha_0 < 1$, that is, 
		$$\int_{\mathbb{R}^n}Q^{(n)}_g\ud \mu_g< \frac{(n-1)!|\mathbb{S}^n|}{2}.$$
		For each fixed point $p_0$, each natural number $k\leq \frac{n-2}{2}$ and  and any real number $1\leq p<\frac{n-1}{2k}$, there holds
		$$\limsup_{r\to\infty}r^{2kp-n}\int_{B_r^g(p_0)}|Q_g^{(2k)}|^p\ud\mu_g <+\infty.$$
		where $B_r^g(p_0)$ is the geodesic ball  respect to metric $g$ with radius $r$ centered at the point $p_0$. 
	\end{theorem}    
	
	\begin{proof}
		By setting up the coordinate system, we may assume without loss of generality that $p_0 = 0$.  Since $g$ is complete, as a set in $\mr^n$, $\overline {B^g_r(0)}$ must be compact. Thus there exists a point $y_r \in \partial {B^g_r}(0)$ such that  it will maximize the Euclidean norm:
		$$ |y_r| \geq |x| \quad \text{for all } x \in B^g_r(0). $$
		This immediately implies the inclusion
		\begin{equation}\label{geodesic ball included in E ball}
			B^g_r(0) \subset B_{|y_r|}(0).
		\end{equation}
		Moreover, by taking $r$ sufficiently large, we may assume $|y_r| \gg 1$.
		
		Applying Lemma \ref{lem:int_B_r Q^2k)p} together with H\"older's inequality, we obtain
		\begin{equation}\label{intB_r^g Q^2k}
			\int_{B_r^g(0)}|Q_g^{(2k)}|^p\ud\mu_g\leq C+C\left(\int^{|y_r|}_1e^{\bar u(s)}\ud s\right)^{\frac{2kp}{n-1}}\left(\int^{|y_r|}_1 s^{n-1}e^{n\bar u(s)}\ud s\right)^{\frac{n-1-2kp}{n-1}}.
		\end{equation}
		Under the aforementioned assumptions,  Theorem \ref{thm:Wang's theorem}  in conjunction with Lemma \ref{lem:strong Ainfty for bar u} establishes that  both  $e^{nu}$ and $e^{n \bar u}$ are  strong $A_\infty$ weights.  Furthermore, inequality \eqref{d_g leq delta_g} yields
		\begin{equation}\label{d_bar g leq delta bar_g}
			d_{\bar g}(0,y_r)=\int^{|y_r|}_0e^{\bar u}\ud s\leq C\delta_{\bar g}(0, y_r)\leq C\left(\int_{B_{|y_r|}(0)}e^{n\bar u}\ud x\right)^{\frac{1}{n}}.
		\end{equation}
		Making use of Jensen's inequality, one has
		\begin{equation}\label{e^nbar u leq e^nu}
			\int_{B_{|y_r|}(0)}e^{n\bar u}\ud x=\int^{|y_r|}_0|\ms^{n-1}|s^{n-1}e^{n\bar u(s)}\ud s\leq \int_{B_{|y_r|}(0)}e^{n u}\ud x.
		\end{equation}

		Thus, combining the estimates \eqref{d_bar g leq delta bar_g},  \eqref{e^nbar u leq e^nu} with \eqref{intB_r^g Q^2k}, we obtain that 
		\begin{equation}\label{intQ^2k leq e^nu}
			\int_{B_r^g(0)}|Q_g^{(2k)}|^p\ud\mu_g\leq C+C\left(\int_{B_{|y_r|}(0)}e^{n u}\ud s\right)^{\frac{n-2kp}{n}}.
		\end{equation}
		
		Recall that  $e^{nu}$ is a strong $A_\infty$ weight. This, together with \eqref{delta_g leq Cd_g}, implies 
		\begin{equation}\label{delta_g leq Cr}
			\delta_g(0,y_r)\leq Cd_g(0,y_r)=Cr.
		\end{equation}
		Use the doubling property \eqref{doubling property} to conclude:
		\begin{equation*}
			\int_{B_{2|y_r|}(\frac{y_r}{2})}e^{nu}\ud x\leq C\int_{B_{\frac{|y_r|}{2}}(\frac{y_r}{2})}e^{nu}\ud x \leq C\int_{B_{|y_r|}(0)}e^{nu}\ud x.
		\end{equation*}
		Observe that
		\begin{equation*}
			B_{|y_r|}(0)\subset B_{2|y_r|}(\frac{y_r}{2}).
		\end{equation*}
		Thus we have $$ \int_{B_{|y_r|}(0)}e^{nu}\ud x \leq \int_{B_{2|y_r|}(\frac{y_r}{2})}e^{nu}\ud x\leq C\int_{B_{\frac{|y_r|}{2}}(\frac{y_r}{2})}e^{nu}\ud x = C \delta_g(0,y_r)^n.$$
		Combine this estimates with \eqref{delta_g leq Cr} and \eqref{intQ^2k leq e^nu} to obtain: 
		$$	\int_{B_r^g(0)}|Q_g^{(2k)}|^p\ud\mu_g \leq C+Cr^{n-2kp}.$$
		Thus, the proof is complete.		
	\end{proof}

	\hspace{3in} 
	
	Now we are in position to finish the proofs of our theorems \ref{thm:Q^2k_g growth_on_geodesic_ball} and \ref{thm:Yau's conjecture}.
	
	\hspace{3in} 
	
	{\bf Proof of Theorem \ref{thm:Q^2k_g growth_on_geodesic_ball}:}
	
	With the help of Theorem~\ref{thm: int-Q^n-between 0,1}, we can easily conclude that  $Q^{(n)}_g e^{n u} \in L^1(\mathbb{R}^n)$. 
	Then, by Lemma~\ref{lem: normal metric}, we have that $g$ is a normal metric.  
	
	Based on the assumption $I_g > 0$, Corollary~\ref{cor:I_g-positive} shows that
	$$
	\int_{\mathbb{R}^n} Q^{(n)}_g \, d\mu_g < \frac{(n-1)! |\mathbb{S}^n|}{2}.
	$$
	
	Finally, applying Theorem~\ref{thm:Q^2k_g L^p growth for normal metric}, we finish our proof. \qed

	\hspace{3em}

Theorem \ref{thm:Yau's conjecture} is a direct consequence of the following result by using the estimate  \[|\sigma_k(g)|\leq |Ric_g|_g^k.\]
	
\begin{theorem}\label{thm:Ricci growth}
	Under the same assumptions as in Theorem \ref{thm: positive Ricci}, if  we further assume that  the  isoperimetric ratio $I_g$ is positive, for any $1\leq q<\frac{n-1}{2}$ and fixed point $p_0$,  there holds
	\[\lim_{r\to\infty}\sup r^{2q-n}\int_{B_r^g(p_0)}|Ric_g|^q_g\ud\mu_g<+\infty.\]
\end{theorem}
\begin{proof}
		With the same as in the proof of Theorem \ref{thm:Q^2k_g growth_on_geodesic_ball}, $g$ must be a normal metric.   
	Using \eqref{Ricci-equ} and Young's inequality, one has
	\begin{equation}\label{ricci upper bound}
	|Ric_g|_g^q \leq C(|\nabla^2 u|^q+|\nabla u|^{2q})e^{-2qu}.
	\end{equation}
	For any $r\geq 1$, using \eqref{ricci upper bound} and Lemma \ref{lem:equ nabla^k Lf}, one has
	\begin{align*}
		\int_{B_r(0)}	|Ric_g|_g^q\ud\mu_g\leq & C+C\int_1^rC(|\nabla^2 u|^q+|\nabla u|^{2q})e^{(n-2q)u}\ud x\\
		\leq &C+C\int^{r}_1r^{n-2q-1}e^{(n-2q)\bar u(r)}\ud r.
	\end{align*}
	With the help of the same argument as in the proof of Theorem \ref{thm:Q^2k_g L^p growth for normal metric}, we finish our proof. 
	\end{proof}

	\section{Gap theorems for $Q^{(2k)}_g$-curvature}\label{sec:gap theorem}
	
	Before establishing the gap theorems, we need the following crucial potential estimates.

	\begin{lemma}\label{lem:r^2-Deltau}(See Lemma 2.2 in \cite{Li 24 total Q-curvature})
		Consider the normal metric $g=e^{2u}|dx|^2$ on $\mr^n$ of even dimension at least $4$. Then, there holds
		$$	r^2\cdot \fint_{\partial B_r(0)}(-\Delta)u\ud\sigma \to(n-2)\alpha_0, \quad \mathrm{as}\; r\to\infty,$$
		with $\alpha_0$ is given by \eqref{alpha_0}.		
	\end{lemma}

	\begin{lemma}\label{lem:r^2abla-u^2}(See Lemma 2.4 in \cite{Li 24 total Q-curvature})
		Again we consider the normal metric $g=e^{2u}|dx|^2$ on $\mr^n$ with even dimenion at least $4$. Then, there holds
		$$	r^2\cdot \fint_{\partial B_r(0)}|\nabla u|^2d\sigma \to\alpha_0^2, \quad \mathrm{as}\; r\to\infty.$$
	\end{lemma}

	With the help of above two estimates, we are able to obtain the following growth rate  of the average scalar curvature over large  sphere.
	
	\begin{lemma}\label{lem:R_g limit on sphere}
		Let $g=e^{2u}|dx|^2$ on $\mr^n$  with $n\geq 4$  an even integer be a complete and conformal metric. Suppose that both $nth$-order Q-curvature  $Q^{(n)}_g$ and scalar curvature $R_g$ are non-negative. Then there holds
		$$r^2\fint_{\partial B_{r}(0)}R_{g}e^{2u}\ud \sigma\to(n-1)(n-2)\alpha_0(2-\alpha_0),  \quad \mathrm{as}\; r\to\infty. $$
	\end{lemma}
	\begin{proof}
		With the help of Theorem \ref{thm: int-Q^n-between 0,1} and Lemma \ref{lem: normal metric}, the  conformal metric $g$  is normal and has finite total $Q^{(n)}_g$-curvature. 
		Recall that 
		$$R_ge^{2u}=-2(n-1)\left(\Delta u+\frac{n-2}{2}|\nabla u|^2\right).$$
		Using Lemma \ref{lem:r^2-Deltau} and Lemma \ref{lem:r^2abla-u^2}, one has 
		$$r^2\fint_{\partial B_{r}(0)}R_{g}e^{2u}\ud \sigma\to (n-1)(n-2)\alpha_0(2-\alpha_0)$$
		as $r\to\infty.$
		Thus, we finish the proof.
	\end{proof}

	\begin{lemma}\label{lem:|y|^k/|x-y|^kud nu(y)to0}
		For any non-negative and  integral measure $\theta$ on $\mr^n$ with integer $n\geq 3$ and constant $0<p\leq n-2$,  one has  
		$$\fint_{\partial B_{r}(0)}\int_{\mr^n}\frac{|y|^p}{|x-y|^p}\ud\theta(y)\ud\sigma\to 0, \quad \mathrm{as}\; r\to\infty.$$
	\end{lemma}
	\begin{proof}
		For any $\epsilon$, there exists $R_\epsilon>0$ such that 
		$$\int_{\mr^n\backslash B_{R_\epsilon}(0)}\ud\theta(y)<\epsilon.$$
		
		Therefore, for the fixed $R_\epsilon$, we decompose the integral into two parts:
		\begin{align*}
			&\fint_{\partial B_{r}(0)}\int_{\mr^n}\frac{|y|^p}{|x-y|^p}\ud\theta(y)\ud\sigma\\
			=&\fint_{\partial B_{r}(0)}\int_{B_{R_\epsilon}(0)}\frac{|y|^p}{|x-y|^p}\ud\theta(y)\ud\sigma\\
			+&\fint_{\partial B_{r}(0)}\int_{\mr^n\backslash B_{R_\epsilon}(0)}\frac{|y|^p}{|x-y|^p}\ud\theta(y)\ud\sigma
		\end{align*}
		On one hand, 	by the dominated convergence theorem, the first term on the right yields that 
		$$\int_{B_{R_\epsilon}(0)}\frac{|y|^p}{|x-y|^p}\ud\theta(y)\to 0, \quad  \mathrm{as}\; |x|\to\infty.$$
		
		On the other hand, using Fubini's theorem and Lemma \ref{lem: mean value}, the second term on the right behalves like:
		\begin{align*}
			&\fint_{\partial B_{r}(0)}\int_{\mr^n\backslash B_{R_\epsilon}(0)}\frac{|y|^p}{|x-y|^p}\ud\theta(y)\ud\sigma\\
			\leq &\int_{\mr^n\backslash B_{R_\epsilon}(0)}\fint_{\partial B_{r}(0)}\frac{|y|^p}{|x-y|^p}\ud\sigma\ud\theta(y)\\
			\leq &\int_{\mr^n\backslash B_{R_\epsilon}(0)}\ud\nu(y)<\epsilon.
		\end{align*}
		Since $\epsilon$ is arbitrary, we finish our proof.
	\end{proof}

	In order to get control on higher $Q$-curvature estimate, we need to control the higher order derivative. To do so, we follow the argument of \cite{Li 24 total Q-curvature}. First let us do the fourth derivative estimate:
	
	\begin{lemma}\label{lem:r^4Delta^2u}
		Suppose the normal metric $g=e^{2u}|dx|^2$ on $\mr^n$ with even dimension at least $6$ is given. Then, the spherical average of bi-Laplace of $u$ near infinity behaves like:
		$$r^4\fint_{\partial B_r(0)}\Delta^2u\ud \sigma\to 2(n-2)(n-4)\alpha_0, \quad \mathrm{as}\; r\to\infty.$$
	\end{lemma}
	\begin{proof}
		Rewrite the equation \eqref{equ-Delta^2u} to have
		\begin{align*}
			&\frac{|x|^4\Delta^2u(x)}{2(n-2)(n-4)}=\int_{\mr^n}\frac{|x|^4}{|x-y|^4}\ud\nu(y)\\
			=&\int_{\mr^n}\left(\frac{|x-y+y|^2}{|x-y|^2}\right)^2\ud\nu(y)\\
			=&\int_{\mr^n}\left(1+\frac{|y|^2+2y\cdot(x-y)}{|x-y|^2}\right)^2\ud\nu(y)\\
			=&\int_{\mr^n}\left(1+\frac{(|y|^2+2y\cdot(x-y))^2}{|x-y|^4}+2\frac{|y|^2+2y\cdot(x-y)}{|x-y|^2}\right)\ud\nu(y).
		\end{align*}
		
		Now use the notation $|\nu|$ to denote the measure:
		$$\ud|\nu|(y):=\frac{2}{(n-1)!|\mathbb{S}^n|}|Q^{(n)}_g(y)|e^{nu(y)}\ud y.$$
		
		On one hand, one has 
		\begin{equation}\label{|x|^4Delta^2u upper bound}
			\frac{|x|^4\Delta^2u(x)}{2(n-2)(n-4)}\leq \alpha_0+\int_{\mr^n}\left(\frac{10|y|^2}{|x-y|^2}+\frac{4|y|}{|x-y|}+\frac{2|y|^4}{|x-y|^4}\right)\ud|\nu|(y).
		\end{equation}

		On the other hand, one will have
		\begin{equation}\label{|x|^4Delta^2u lower bound}
			\frac{|x|^4\Delta^2u(x)}{2(n-2)(n-4)}\geq \alpha_0-\int_{\mr^n}\left(\frac{10|y|^2}{|x-y|^2}+\frac{4|y|}{|x-y|}+\frac{2|y|^4}{|x-y|^4}\right)\ud|\nu|(y).
		\end{equation}
		Using \eqref{|x|^4Delta^2u upper bound}, \eqref{|x|^4Delta^2u lower bound} and Lemma \ref{lem:|y|^k/|x-y|^kud nu(y)to0}, one has 
		$$r^4\fint_{\partial B_r(0)}\Delta^2u\ud\sigma\to 2(n-2)(n-4)\alpha_0.$$
	\end{proof}

	\begin{lemma}\label{lem:r^4(Deltau)^2}
		Consider the normal metric $g=e^{2u}|dx|^2$ on $\mr^n$ with the even dimension at least $6$. Then, the average of square of Laplace of $u$ has the growth rate:
		$$r^4\fint_{\partial B_r(0)}(\Delta u)^2\ud\sigma \to  (n-2)^2\alpha_0^2,  \quad \mathrm{as}\; r\to\infty. $$
	\end{lemma}
	\begin{proof} We follow the proof mode of previous lemma: first we rewrite it as:
		\begin{align*}
			\frac{|x|^4(\Delta u(x))^2}{(n-2)^2}=&\left(\int_{\mr^n}\frac{|x|^2}{|x-y|^2}\ud\nu(y)\right)^2\\
			=&\left(\int_{\mr^n}\left(1+\frac{|y|^2+2y\cdot(x-y)}{|x-y|^2}\right)\ud\nu(y)\right)^2\\
			=&\alpha_0^2+2\alpha_0\int_{\mr^n}\frac{|y|^2+2y\cdot(x-y)}{|x-y|^2}\ud\nu(y)\\
			&+\left(\int_{\mr^n}\frac{|y|^2+2y\cdot(x-y)}{|x-y|^2}\ud\nu(y)\right)^2
		\end{align*}
		
		Cauchy and H\"older's inequalities provide us the following estimate:
		\begin{equation}\label{r^4Deltau^2 upper bound}
			\frac{|x|^4(\Delta u(x))^2}{(n-2)^2}\leq \alpha_0^2+C\alpha_0\int_{\mr^n}\left(\frac{|y|}{|x-y|}+\frac{|y|^2}{|x-y|^2}+\frac{|y|^4}{|x-y|^4}\right)\ud|\nu|(y).
		\end{equation}
		and 
		\begin{equation}\label{r^4Deltau^2 lower bound}
			\frac{|x|^4(\Delta u(x))^2}{(n-2)^2}\geq \alpha_0^2-C\alpha_0\int_{\mr^n}\left(\frac{|y|}{|x-y|}+\frac{|y|^2}{|x-y|^2}+\frac{|y|^4}{|x-y|^4}\right)\ud|\nu|(y).
		\end{equation}
		
		Combining the estimates \eqref{r^4Deltau^2 upper bound}, \eqref{r^4Deltau^2 lower bound} with Lemma \ref{lem:|y|^k/|x-y|^kud nu(y)to0}, one has 
		$$r^4\fint_{\partial B_r(0)}(\Delta u)^2\ud\sigma \to (n-2)^2\alpha_0^2.$$
		
	\end{proof}
	
	\begin{lemma}\label{lem:fint Q^4 lower limit}
		Let $g=e^{2u}|dx|^2$ on $\mr^n$  with $n\geq 6$  an even integer be a complete and conformal metric. Suppose that both n-order Q-curvature  $Q^{(n)}_g$ and scalar curvature $R_g$ are non-negative. Then, there holds
		\begin{eqnarray*}
			& \lim_{r\to\infty}\inf r^4\fint_{\partial B_r(0)}Q^{(4)}_ge^{4u}\ud\sigma \nonumber\\
			\geq & \frac{(n-4)^4}{16} \left(2-\alpha_0\right)\left(\frac{2(n-2)}{n-4}-\alpha_0\right)\left(\alpha_0+\frac{4}{n-4}\right)\alpha_0.
		\end{eqnarray*}
	\end{lemma}
	\begin{proof}
		With the help of Theorem \ref{thm: int-Q^n-between 0,1} and Lemma \ref{lem: normal metric}, the  conformal metric $g$  is normal and has finite total $Q^{(n)}_g$-curvature. 
		Combing Lemmas \ref{lem:r^4Delta^2u}, \ref{lem:r^4(Deltau)^2} with \eqref{positive Q^4}, one has
		\begin{align*}
			&\frac{4}{(n-4)^2}\lim_{r\to\infty}\inf r^4\fint_{\partial B_r(0)}Q^{(4)}_ge^{4u}\ud\sigma\\
			\geq&(2-\alpha_0)(\frac{2(n-2)}{n-4}-\alpha_0)\lim_{r\to\infty}\inf r^4\fint_{\partial B_r(0)}\left(\frac{(n-4)^2}{4(n-2)^2}(\Delta u)^2+\frac{\Delta^2 u}{2(n-2)}\right)\ud\sigma\\
			=&(2-\alpha_0)(\frac{2(n-2)}{n-4}-\alpha_0)\frac{(n-4)^2}{4}(\alpha_0+\frac{4}{n-4})\alpha_0.
		\end{align*}
		Thus, we finish the proof.
	\end{proof}

	\vspace{1em}
	
	{\bf Proof of Theorem \ref{thm:Q^2k_g gap theorem} and Theorem \ref{thm:Q^2k_g gap theoremfor logr}:}
	
	\vspace{1em}
	
	Firstly, Theorem \ref{thm: int-Q^n-between 0,1} can be applied to show that $Q^{(n)}_ge^{nu}\in L^1(\mr^n)$ and $g$ is a normal metric. Moreover,  there holds
	$$0\leq \alpha_0\leq 1.$$
	Since $Q^{(n)}_g\geq 0$, if $\alpha_0=0$, one must have $Q^{(n)}_g\equiv 0$ and then $u$ must be a constant due to   the equation \eqref{normal-metric-def}.

	We firstly deal with the case $I_g>0$. If $I_g>0$, one has $\alpha_0<1$ due to Corollary \ref{cor:I_g-positive}.
	Next, we argue with contradiction. If $0<\alpha_0<1$, we claim that for $1\leq k\leq \min\{2,\frac{n-2}{2}\}$ and each small constant  $0<\epsilon<n-2k$, there holds
	$$\int_{B_r^g(0)}Q^{(2k)}_g\ud \mu_g \geq C_1r^{n-2k-\frac{\epsilon}{2}}, \quad \mathrm{for} \; r\gg1$$
	where $C_1$ is a positive constant.

	We have already shown that, in Theorem \ref{thm:positive Q_4}, that $Q^{(2k)}_g$ and $Q^{(2)}_g$ are both non-negative.
	Once again we can choose $z_r\in\partial B_r^g(0)$ such that $|z_r| = \min\{ |z|  | z \in \partial B_r^g(0)\}$ , therefore $  B_{|z_r|}(0) \subset B_r^g(0)$.  Since $g$ is complete, $|z_r|$ must tend to infinity if $r\to\infty.$
	Then, there holds
	\begin{align*}
		&\int_{B_r^g(0)}Q^{(2k)}_g\ud \mu_g \\
		\geq &\int_{B_{|z_r|}(0)}Q^{(2k)}_ge^{nu}\ud x \\
		= &\int^{|z_r|}_0\int_{\partial B_s(0)}Q^{(2k)}_ge^{nu}\ud\sigma\ud s.
	\end{align*}
	
	Since $Q^{(n)}_g\geq 0$,    Lemma 2.10 in \cite{Li 23 Q-curvature} yields that 
	$$u(x)\geq -\alpha_0\log |x|-C,\quad  \mathrm{for}\; |x|\geq 1$$
	
	Then, insert this lower bound into above estimate to get
	\begin{equation}\label{intB_z_r-Q^2k-lower-bound}
		\int_{B_r^g(0)}Q^{(2k)}_g\ud \mu_g\geq C\int^{|z_r|}_1s^{-(n-2k)\alpha_0}\int_{\partial B_s(0)}Q^{(2k)}_ge^{2ku}\ud\sigma\ud s.
	\end{equation}
	Make use of Lemma \ref{lem:R_g limit on sphere} for $k = 1$ and Lemma \ref{lem:fint Q^4 lower limit} for $ k = 2$, together with the assumption $0<\alpha_0<1$ on the estimate \eqref{intB_z_r-Q^2k-lower-bound}, to reach the following estimate:
	$$\int_{B_r^g(0)}Q^{(2k)}_g\ud \mu_g\geq C\int^{|z_r|}_1s^{-(n-2k)\alpha_0}s^{n-1-2k}\ud s=C\frac{|z_r|^{(n-2k)(1-\alpha_0)}-1}{(n-2k)(1-\alpha_0)}.$$
	
	Using the geodesic distance comparison identity \eqref{distance comparison identity},  one has
	$$\lim_{|x| \to \infty}\frac{\log d_g(x,0)}{\log|x|}=1-\alpha_0.$$
	Thus, for any $\delta > 0$,  there will be a $R_\delta > 0$ such that for all $r > R_\delta$, one has
	$$|z_r|^{1-\alpha_0+\delta}\geq d_g(z_r,0)=r.$$
	Thus, if we make a choice of  $\delta = \frac{(1-\alpha_0)\epsilon}{2(n-2k)-\epsilon}>0$, then we reach an inequality which contradicts with our assumption.  This contradiction shows that one must have $\alpha_0=0$. Thus,  we finish the proof of Theorem \ref{thm:Q^2k_g gap theorem}.

	Finally, we turn to  the proof of Theorem \ref{thm:Q^2k_g gap theoremfor logr}.
	Without assumption on $I_g$, we just need to rule out the case $\alpha_0=1$. In this situation, using \eqref{intB_z_r-Q^2k-lower-bound} and Lemmas \ref{lem:fint Q^4 lower limit}, \ref{lem:R_g limit on sphere} as before, one has
	$$\int_{B_r^g(0)}Q^{(2k)}_g\ud\mu_g\geq C\int^{|z_r|}_1s^{-1}\ud s=C\log|z_r|.$$
	Using the above estimate and \eqref{distance comparison identity}, one has
	$$\frac{\int_{B_r^g(0)}Q^{(2k)}_g\ud\mu_g}{\log r}\geq C\frac{\log|z_r|}{\log r}=C\frac{\log |z_r|}{\log d_g(0,z_r)}\to\infty,\quad  \mathrm{as}\; r\to\infty.$$
	which contradicts to our assumption. Thus, we finish the proof of Theorem \ref{thm:Q^2k_g gap theoremfor logr}.
	\qed

	\section{Some remarks on the vanishing isoperimetric ratio}\label{sec:vanishing I_g}
	
Combining Theorem \ref{thm: int-Q^n-between 0,1} with Theorem 1.2 in \cite{Li 24 total Q-curvature}, we conclude that equality holds on the right-hand side of \eqref{CVineqaulity for Q} if the scalar curvature is bounded by some positive constant from below. In this situation, applying Theorem \ref{thm:CQY-thm}, we find that the isoperimetric ratio $I_g$ must vanish.   
	
	Due to such observation, we  give the following geometric criterion for vanishing $I_g$.
	
	\begin{theorem}\label{thm:I_g vanishing}
		Let $g=e^{2u}|dx|^2$ be a complete and conformal metric on $\mr^n$ with an even integer  $n\geq4$. Suppose that scalar curvature $R_g$ is non-negative  outside a compact set  and the negative part of $Q^{(n)}_g$ is integrable over $(\mr^n,g)$. 
		For each integer  $1\leq k\leq \frac{n-2}{2}$, if $2kth$-order $Q^{(2k)}_g\geq C>0$ for some positive constant $C$ outside a compact subset, then the iso-perimetric ratio $I_g$ vanishes.
	\end{theorem}
	\begin{proof}
		With the help of Theorem \ref{thm: int-Q^n-between 0,1}, we know that $Q^{(n)}_g$ is absolutely integrable and $g$ is a normal metric.
		For each given $p>0$,   apply  Lemma \ref{lem:e^ku=ek bar u} with  $f=-Q^{(n)}_ge^{nu}$  to get the estimate:
		\begin{equation}\label{int e^-pu}
			\fint_{\partial B_r(0)}e^{-pu}\leq Ce^{-p\bar u}.
		\end{equation}
		The same argument for Lemma \ref{lem:int_B_r Q^2k)p} , with the help of  \ref{int e^-pu}, for each natural number $k\leq \frac{n-2}{2}$ and sufficiently large $r$, shows that
		\begin{align*}
			\fint_{\partial B_r(0)}Q^{(2k)}_g\ud\sigma\leq &\fint_{\partial B_r(0)} e^{-\frac{n+2k}{2}u}|\Delta^ke^{\frac{n-2k}{2}u}|\ud\sigma\\
			\leq &C\fint e^{-2ku}\sum^{2k}_{i=1}|\nabla^i u|^{\frac{2k}{i}}\ud\sigma\\
			\leq &Cr^{-2k}e^{-2k\bar u(r)}.
		\end{align*}
		If $Q^{(2k)}_g\geq C>0$ outside a compact set,  combine  Lemma  \ref{lem:barLf} and this desired estimate to conclude that
		$$\alpha_0\geq 1.$$
		On the other hand the equation \eqref{CVineqaulity for Q} already shows $\alpha_0 \leq 1$. Thus we have to have:
		$$\alpha_0=1.$$
		Immediately, Theorem \ref{thm:CQY-thm} yields that $I_g=0.$
		
		Thus, we completes our argument.
	\end{proof}

\end{document}